\documentclass[english]{IEEEtran} 
\usepackage[T1]{fontenc}
\usepackage[latin9]{inputenc}
\usepackage{float}
\usepackage{amsmath}
\usepackage{amsthm}
\usepackage{amssymb}
\usepackage{graphicx}
\usepackage{esint}
\usepackage{epstopdf}

 \usepackage{pgfplots}
  \pgfplotsset{compat=newest}
  \usetikzlibrary{plotmarks}
  \usetikzlibrary{arrows.meta}
  \usepgfplotslibrary{patchplots}
  \usepackage{grffile}
  \usepackage{amsmath}

\usepackage{epsfig}
\usepackage{subfigure}
\usepackage{amssymb,amsmath,amsfonts,layout,graphicx}
\usepackage{makeidx}
\usepackage{babel}
\usepackage{tikz}
\usepackage{sublabel}
\usepackage{cases}
\usepackage{algorithm}
\usepackage{algorithmic}

\usepackage
[
        letterpaper,
        left=1.91cm,
        right=1.91cm,
        top=1.91cm,
        bottom=2cm,
]
{geometry}

\newtheorem{lemma}{Lemma}
\newtheorem{assumption}{Assumption}
\newtheorem{proposition}{Proposition}
\newtheorem{theorem}{Theorem}
\newtheorem{definition}{Definition}

\newtheorem{remark}{Remark}
\newtheorem{axiom}{Axiom}

 \title{Optimal Resource Procurement and \\ the  Price of Causality}	

\author{Sen Li$^\dagger$, Akhil Shetty$^\star$, Kameshwar Poolla$^\star$, and Pravin Varaiya$^\star$

 \thanks{This research is supported by the National Science Foundation under grants EAGER-1549945 and CPS-1646612, and by the National Research Foundation of Singapore under a grant to the Berkeley Alliance for Research in Singapore. The conference version of this paper was presented at 2018 European Control Conference \cite{shetty2018optimal}.  }
 \thanks{$^\dagger$The Department of Mechanical Engineering, University of California, Berkeley, USA. Email: lisen1990@berkeley.edu}
  \thanks{$^\star$The Department of Electrical Engineering and Computer Science, University of California, Berkeley, USA. Email: \{shetty.akhil, poolla, varaiya\}@berkeley.edu}

}

\begin{document}
\maketitle
\begin{abstract}
This paper studies the problem of procuring diverse resources in a forward market to cover a set $\bf{E}$ of uncertain demand signals $\bf{e}$. We consider two scenarios: (a) $\bf{e}$ is revealed all at once by an oracle (b) $\bf{e}$ reveals itself causally. 
Each scenario induces an optimal procurement cost. The ratio between these two costs is defined as the {\em price of causality}. It captures the additional cost of not knowing the future values of the uncertain demand signal. We consider two application contexts:  procuring energy reserves from a forward capacity market, and purchasing virtual machine instances from a cloud service. An upper bound on the price of causality is obtained, and the exact price of causality is computed for some special cases. The algorithmic basis for all these computations is set containment linear programming.  A mechanism is proposed to allocate the procurement cost to consumers who in aggregate produce the demand signal. We show that the proposed cost allocation is fair, budget-balanced, and respects the cost-causation principle. The results are validated through numerical simulations.

\end{abstract}
\section{Introduction}
Many  complex systems consist of both controllable and uncontrollable resources. Examples range from  power systems, computing systems, transportation systems, among others.  Uncontrollable resources inject uncertainties in the system, which collectively generate an uncertain demand signal that needs to be balanced by the controllable resources. We investigate optimal resource procurement necessary to balance the uncertain aggregate demand signal.  
One example is that of a grid operator that needs to procure energy reserves from the forward capacity market. He can choose from diverse resources, such as batteries, generators, aggregation of loads, etc. These resources are used to cover imbalances between electricity supply and demand for the grid. 
Another example is that of a company buying virtual machines from a cloud provider to serve computational demands, which are typically heterogeneous.  The common thread that binds these examples is the {\em ex-ante} procurement of diverse resources to cover an uncertain signal revealed in real-time.

What is the optimal resource asset mix that covers all uncertain demand signals $e \in E$? The purchase decision depends on the unit prices of the resources, their dynamic constraints, and critically on the control strategy that allocates the signal to the procured resources. The resource procurement decision is therefore intimately coupled with the real-time control strategy associated with allocation. Since the signal reveals itself in real-time, the allocation policy needs to be causal. The operator needs to irrevocably  commit procured resources to match the uncertain demand $e$ without the luxury of knowing its future values. Clearly, the resource procurement cost under causal policies will be higher than that under arbitrary(possibly non-causal) policies. This inspires us to quantify the effect of causality on  optimal resource procurement cost. 

A second issue is that of paying for the procurement cost. Ideally, the cost allocation should be fair, budget-balanced, and follow the cost causation principle. This principle enunciates that agents are penalized (rewarded) in proportion to their contribution (mitigation) of the need to procure balancing resources. We explore cost allocation mechanisms that satisfy these requirements.

\subsection{Our Contribution}
We formulate optimal resource procurement as a set containment problem. In particular, consider an operator that  procures diverse resources to  cover a sequence of uncertain signals revealed over a delivery window of length $T$.  Each resource has linear dynamic constraints, which can be modeled as a convex set in $R^T$. Uncertain signals are modeled as belonging to a specified convex set $E$. The operator needs to determine  the optimal resource mix to collectively cover all signals in $E$. The principle contributions of the paper are:
\begin{itemize}
\item We define the  {\em price of causality} (PoC). It is the ratio of the optimal procurement cost under  causal policies to that under  arbitrary (possibly non-causal) policies.  It quantifies the additional cost for not knowing the information in the future.

\item We derive an upper bound on the price of causality. The algorithmic basis of this computation is set containment linear programming.

\item We obtain the exact price of causality for some special cases. Through these cases, we show that {\em dynamics} and {\em diversity} are the main factors that drive the price of causality to be greater than 1.

\item A cost allocation mechanism is proposed. It is fair, budget balanced, and respects the cost causation principle. 

\item We illustrate our framework in two application contexts:  procuring energy reserves from forward capacity markets, and buying virtual machines from cloud providers.

\end{itemize}

\subsection{Related Work}
A closely related problem is online optimization \cite{borodin2005online}, \cite{buchbinder2009design}. It studies sequential decisions made irrevocably at each time step without access to future information. This problem is widely studied in many areas, such as stochastic dynamic programs \cite{brown2014information}, communication networks \cite{mao2016optimal}, online allocation \cite{hao2017online}, \cite{bhalgat2012online}, load balancing \cite{lin2012online}, among others. A standard measure to evaluate the performance of an online solution is the competitive ratio \cite{borodin2005online}. This compares the performance of the optimal online algorithm, where all information is revealed causally, to the performance of the optimal offline algorithm, which is an unrealizable algorithm with complete information about the future.
Online optimization problems consider the available resources to be fixed and aim to find the optimal causal decisions that minimize \emph{real-time costs}. Our problem is distinct. We consider resource procurement problems that minimize ex-ante \emph{capacity cost}. 
This distinction differentiates the price of causality from the competitive ratio.

Aside from online optimization, another strand of related work studies the adequacy of resources in real-time decision making. For instance, Dertouzos {\em et al.} studies the online processor time allocation problem in \cite{dertouzos1989multiprocessor}, and shows that optimal scheduling is impossible without a priori knowledge on the start times of tasks.  Subramanian {\em et al.} considers a real-time scheduling problem for distributed energy resources to reduce the grid energy cost \cite{subramanian2013real}. Wenzel {\em et al.} studies real-time charging strategies for electric vehicles to provide ancillary services with minimum tracking error \cite{wenzel2017real}.  Madjidian {\em et al.} discovers the trade-off between absorbing and releasing energy for collective loads under causal allocation policies \cite{madjidian2017energy}. In all of these works, the quantity of available resources is fixed. Their focus is on analysis of causal policies. In contrast, we investigate optimal resource procurement to meet worst case adequacy.   

The closest related works are  
 \cite{nayyar2016duration} and \cite{negrete2016rate}. Negrete-Pincetic {\em et al.} considers a supplier who owns uncertain renewable generations and also purchases energy in day-ahead and real-time markets to serve the deferrable loads \cite{negrete2016rate}. They show that the optimal procurement costs under causal allocation policy and offline allocation policy are the same, i.e., the price of causality is 1. In \cite{nayyar2016duration} and \cite{negrete2016rate},  the supplier can purchase additional energy from the real-time markets when the day-ahead procurement is not enough. This is distinct from our problem, where  all procurement decisions are made in advance, and no recourse is available in real time.

The remainder of this paper is organized as follows. The resource procurement problem is formulated in Section II, followed by some  examples in Section III.  Section IV and Section V present an upper bound for the price of causality and study some special cases. Section VI discusses the cost allocation mechanism, followed by numerical studies in Section VII. Concluding remarks and future directions are offered in Section VIII.

\subsection{Notation}

Throughout the paper, $\mathbb{R}$ denotes the set of real numbers. $N$ denotes the total number of resources. $T$ denotes the time horizon. For any positive integer $Z$, $[Z]$ denotes the set $\{1, \dots, Z\}$. If $\delta\in \mathbb{R}$ and $A$ is a set, then $\eta A=\{\eta a |a\in A \}$. For two sets $A$ and $B$,   $A=B$ means $A\subseteq B$ and $A \supseteq B$. $A \oplus B$ denotes their Minkowski sum, i.e., $A\oplus B=\{a+b| a\in A, b\in B\}$.
\section{Optimal Resource Procurement}

\subsection{Problem Setup}
Consider two types of resources: controllable resources, and uncontrollable resources. Uncontrollable resources generate an uncertain demand signal, which must be balanced by controllable resources over a delivery window. We segment this delivery window into $T$ contiguous periods, and denote $e=(e^1,\ldots,e^T)$ as the demand signal over these periods. We model the uncertain demand signal $e$ as being contained in the set $E\subset \mathbb{R}^T$, and make the following assumption:
\begin{assumption}
\label{asusmptiononpolytope}
$E$ is a bounded convex polytope in $\mathbb{R}^T$. 
\end{assumption}
In many applications, the demand signal $e$ is modeled stochastically. In our case, the polytope $E$ can be interpreted as the support of the distribution of $e$, or as the confidence interval such that $e\in E$ with probability $1-\epsilon$. In the case study section, we will give an example of how to derive $E$ based on real data.

To balance the uncertain demand signal, an operator chooses from a group of $N$ controllable resources. These resources have diverse prices and dynamic constraints. Resource $i$ can generate a sequence of outputs over the horizon $[1, \ldots, T]$, denoted as $s_i=(s_i^1,\ldots,s_i^T)$.  Let  $s_i\in S_i$,  where $S_i$ is the set of all possible output sequences constrained by the resource dynamics. We make the following assumption:
\begin{assumption}
\label{asusmptiononresource}
$S_i$ is a bounded and convex polytope in $\mathbb{R}^T$, and $0$ is in the interior of $S_i$ for all $i\in [N]$.
\end{assumption}
Assumption \ref{asusmptiononresource} holds if the  dynamic constraints of the resources are linear.  Note that $0\in \text{int }(S_i)$ trivially holds as the  resource can be idle over the delivery window. 

We refer to $S_i$ as a {\em unit resource}. Unit resource   $S_i$ is offered at price $\pi_i$.  If the operator purchases $\alpha_i$ units of resource $i$, he pays $\alpha_i\pi_i$, and has the right to command any signal $s_i\in \alpha_iS_i$ in the delivery window.

The time-line of the problem is shown in Figure \ref{fig:timeline}. At time $t=0$, the operator purchases the minimum-cost asset mix of controllable resources in a forward market.  During the delivery window $[1, \ldots, T]$, the time sequence of the demand signal $e$ is revealed causally (one sample at a time), and the operator dispatches procured resources to match the demand signal. The operator pays a capital cost for the procured resources, but does not pay for subsequent use of these resources during dispatch. 

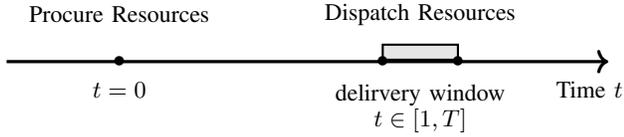
\begin{figure}[t]
\centering
\begin{tikzpicture}[scale=0.5]
{\small
\draw [->,very thick,black] (0,-0.25) --  (16,-0.25);
\node [align=center] at (15.5,-1) {Time $t$};
\node [align=left] at (3.0,1) {Procure Resources}; 
\node [align=center] at (3,-1) {$t=0$}; 
\node [align=center] at (11,1) { Dispatch Resources}; 
\node [align=center] at (11,-1.5) {delirvery window \\
$t\in [1,T]$};  
\node [align=center] at (3,-0.25) {$\bullet$};
\node [align=center] at (10,-0.25) {$\bullet$};
\node [align=center] at (12,-0.25) {$\bullet$};}
\draw [black, fill=black, fill opacity = 0.1, thick] (10,-0.2) rectangle (12,0.2);
\end{tikzpicture}
\caption{Time line of the resource procurement problem.} \label{fig:timeline}
\end{figure}

The optimal resource procurement problem is:
\begin{equation}
\label{resourceprocurement}
\hspace{-2.5cm} J^*=\min_{\alpha_1,\ldots,\alpha_N} \sum_{i=1}^{N} \alpha_i \pi_i
\end{equation}
\begin{subnumcases}{\label{constraint4non_causal}}
E \subseteq \alpha_1 S_1 \oplus \alpha_2 S_2 \oplus \cdots \oplus \alpha_N S_N, \label{polytopecontain} \\
\alpha_i\geq 0,  \quad \forall i\in [N].
\end{subnumcases}
The polytope containment constraint (\ref{polytopecontain}) requires all demand signals $e\in E$ be covered by the procured resource mix. Note that (\ref{resourceprocurement}) always has a solution as $0$ is an interior point of $S_i$ (see Assumption \ref{asusmptiononresource}).

\subsection{Energy Reserve Procurement}
\label{energyprocurementsection}

Consider a forward reserve market, where a system operator procures energy reserves to balance the supply and demand in electricity. We categorize the assets in power systems as controllable resources and uncontrollable resources. These resources are modeled as follows:

\subsubsection{Controllable Resources}
Controllable resources include generators, batteries, aggregation of thermostatically controlled loads \cite{hao2015aggregate}, \cite{zhao2017geometric}, among others.  They can provide electricity on demand to balance supply and demand. 
Consider $N$ types of controllable resources in a forward reserve market. Resource $i$ can produce a power sequence $s_i$ during the delivery window $[1, \ldots, T]$. The sequence  $e_i$ is confined by the dynamics. Assume the dynamic constraints of each resource are linear, then $S_i$ is a polytope.  
As an example, consider a battery with capacity constraint $C_i$, charge rate constraint $\bar{r}_i$ and discharge rate constraint $\underline{r}_i$. The power sequence $s_i$ is constrained by:
\begin{align}
\label{batterymodel}
\begin{cases}
\underline{r}_i\leq s_i^t \leq \bar{r}_i, \quad \forall t\in [T],  \nonumber \\
0\leq \theta_iC_i+\sum_{k=1}^t s_i^k\leq C_i, \quad \forall t\in [T], i\in [N], 
\end{cases}
\end{align}
where $\theta_i$ is the initial state of charge. It can be easily verified that $S_i$ is a polytope. It satisfies Assumption \ref{asusmptiononresource}.

\subsubsection{Uncontrollable Resources}
Examples of uncontrollable resources include wind farms, solar panels, and random loads. These resources cannot be dispatched by the system operator. They inject uncertainty to the system, and create imbalances between the supply and demand. 
 
Consider a two-settlement electricity market that consists of a day-ahead market and a real-time market. Each uncontrollable resource trades in the day-ahead market based on the forecast of electricity consumption (production). The forecast error creates an imbalance between supply and demand in the real-time market.  We model the imbalance signal as a vector $e\in \mathbb{R}^T$, and assume it takes values in  a polytope $E\subset \mathbb{R}^T$. It can be viewed as support of the distribution of $e$, or a confident interval so that $e\in E$ with high probability.

\subsubsection{Balancing}

The reserve procurement problem is to determine the asset mix of controllable resources so that all  $e\in E$ are covered. This reduces to an instance of (\ref{resourceprocurement}).

\subsection{Cloud Computing}
Consider a company that procures virtual machine {\em instances} from a cloud provider (such as Amazon EC2) to serve workloads from users. Two types of workloads are considered: (a) transactional workloads, (b) non-interactive batch workloads.  Transactional workloads such as web applications are highly unpredictable and require immediate response. On the other hand,  non-interactive batch workloads can be predicted and only require completion within a specified time frame \cite{garg2011sla}. For example, a financial institution uses transactional workloads to trade stocks and query indices, and uses non-interactive workloads to analyze investment portfolio and model stock performance \cite{carrera2008enabling}. 

The company purchases {\em on-demand instances} from cloud providers to serve these workloads. On-demand instances can be requested at any time. It is charged in a pay-as-you-go manner with a specified {\em price per unit time}. For instance, Amazon ECS adopts a price-per-hour policy that rounds up partial hours of usage \cite{amazonec2}. Note that new instances cannot be initialized instantaneously. Typically, there is a delay of several minutes \cite{li2010cloudcmp} due to hardware resource allocation and the boot of new systems. Therefore, the company should procure enough instances in advance, instead of reactively purchasing instances after the workload arrives. This motivates the following problem: how many instances are enough to serve all workloads over a  horizon (e.g., 20 minutes)?

To tackle this problem, we model the instances,  non-interactive workloads  and transactional workloads as follows:

\noindent{Instances:} without loss of generality, we assume all instances are homogeneous, and the length of each period is $1$. We model an instance as a unit resource. The output of the resource is $s_1^t$, which represents the fraction of time the instance works during the period $t$. Accordingly, $s_1^t$ satisfies the following constraint:
\begin{equation}
\label{feasibleinstance}
0\leq s_1^t \leq 1, \quad \forall t\in [T],
\end{equation}
Let $s_1=(s_1^1,\ldots,s_1^T)$, and define $S_1=\{s_1|0\leq s_1^t \leq 1, \forall t\in [T]\}$.  Clearly, $S_1$ is a hyper-box.

\noindent{Non-Interactive Batch Workloads:} we model the collection of non-interactive workloads as a non-scalable resource.  In particular, consider $M$ non-interactive workloads. Each of them has an arrival time $a_m$, a deadline $d_m$ and a computation time $c_m$, such that $c_m\leq d_m-a_m$. Assume the workloads can be interrupted, and  each workload can be assigned to at most one instance at a time. Let $r_m^t$ denote the fraction of instance computation time allocated to workload $m$ over the period $t$, then we have $0\leq r_m^t \leq 1$.  Denote $s_2^t$ as the total fraction of instance computing time {\em provided} by all batch workloads. Since the workloads {\em demand} computing resources, $s_2^t$ is negative, and satisfies the following constraints:
\begin{subnumcases}{\label{constraintofS4CC}}
s_2^t=-\sum_{m=1}^M r_m^t, \quad \forall t\in [T] \label{negativeconst} \\
\sum_{t=1}^T r_m^t =c_m,  \quad \forall m\in [M], \label{negativeconst2}\\
0\leq r_m^t\leq 1,  \quad\forall a_m\leq t \leq d_m, \\
r_m^t=0,  \quad \quad \forall t<a_m, \quad \forall t>d_m. \label{negativeconst4}
\end{subnumcases}
where $[M]=\{1,\ldots,M\}$, and (\ref{negativeconst2})-(\ref{negativeconst4}) ensures all workloads to be completed by the deadline. Let $s_2=(s_2^1,\ldots,s_2^T)$, and define $S_2$ as the set of $s_2$ that satisfies (\ref{constraintofS4CC}). It can be verified that $S_2$ is a bounded and convex polytope.

\noindent{Transactional Workloads:} we model transactional workloads as a sequence of uncertain signals. At time $t$, all transactional workloads together require $e^t$ amount of instance computation time. Let $e=(e^1,\ldots,e^T)$. Assume that $e$ only takes value in a polytope  $E \subset \mathbb{R}^T$. In this case, both $S_1$, $S_2$ and $E$ are polytopes. Assumption \ref{asusmptiononpolytope} and Assumption \ref{asusmptiononresource} are satisfied.

The resource procurement problem is as follows:
\begin{equation}
\label{resourceprocurementexmaplecc}
\hspace{-1cm} \min_{\alpha_1} \alpha_1
\end{equation}
\begin{subnumcases}{\label{constraint4non_causalexamplecc}}
E \subseteq \alpha_1 S_1 \oplus S_2  \label{polytopecontainexamplecc} \\
\alpha_1\geq 0.
\end{subnumcases}
where $S_1$ and $S_2$ are defined as (\ref{feasibleinstance}) and (\ref{constraintofS4CC}), respectively.

\section{The Oracle Case}
This section studies the optimal resource procurement problem (\ref{resourceprocurement}) under oracle information. We characterize the exact solution to (\ref{resourceprocurement})  as a linear program. We also pinpoint the difficulty of implementing this solution  due to the causal revelation of the uncertain demand $e$.

We first note that since $S_i$ is convex, the controllable resources cover all uncertain signals in $E$ if and only they cover all extreme cases of $E$. As $E$ is a polytope, these extreme cases correspond to its vertices.  Therefore, the set containment constraint (\ref{polytopecontain}) is equivalent to requiring that all vertices of $E$ be contained in the Minkowski sum  $\alpha_1S_1\oplus \cdots \oplus \alpha_NS_N$. If we represent $E$ as the convex hull of its vertices, i.e., $E= \text{conv }(v_1,\ldots,v_K)$, then (\ref{polytopecontain}) is equivalent to: 
\begin{equation}
\label{intermediate}
v_k\in \alpha_1 S_1 \oplus \alpha_2 S_2 \oplus \cdots \oplus \alpha_N S_N , \quad \forall  k\in [K].
\end{equation}
If we represent each resource $S_i$ as the intersection of half-spaces $S_i=\{s_i\in \mathbb{R}^T| A_is_i\leq B_i\}$,  the optimal resource procurement problem becomes:
\begin{theorem}
\label{noncausalresourceprocurement}
The resource procurement problem (\ref{resourceprocurement}) is equivalent to:
\begin{equation}
\label{resourceprocurement_linear}
\hspace{-4cm} \min \sum_{i=1}^{N} \alpha_i \pi_i
\end{equation}
\begin{subnumcases}{\label{constraint4causalcc}}
v_k=\sum_{i=1}^N q_{i,k}, \quad \forall k\in [K], \label{conservationofall_linearcc}\\
A_iq_{i,k}\leq \alpha_iB_i,  \quad \forall k\in [K], i\in [N], \label{feasibleset_icc}\\
\alpha_i\geq 0,  \quad \forall i\in [N].
\end{subnumcases}
The decision variables are $\alpha_1, \ldots, \alpha_N$ and $q_{i,k}\in \mathbb{R}^T$, $\forall i,k$.
\end{theorem}
The proof is deferred to the Appendix. Theorem \ref{noncausalresourceprocurement} asserts that $J^*$ can be determined by solving  the linear program (\ref{resourceprocurement_linear}). This is a joint optimization over the resource asset mix $\alpha$ and  vertex factorizations $q_{i,k}$.

\begin{remark}
Polytopes can be characterized in two ways: intersection of half-spaces (H-representation)  or  convex hull of vertices (V-representation). These representations are equivalent. We have chosen a V-representation for the set $E$ and H-representation for sets $S_i$. We comment that the computational complexity of  (\ref{resourceprocurement}) crucially depends on these choices \cite{mangasarian2002set}.
\end{remark}
\section{The Price of Causality}

\subsection{Causality Matters}
The resource procurement problem (\ref{resourceprocurement}) embeds an underlying resource allocation problem. To realize the solution to (\ref{resourceprocurement}),  each uncertain signal $e\in E$ must be feasibly allocated to the procured resources during the delivery window, i.e., 
\begin{equation}
e=\sum_{i=1}^N s_i, \quad s_i\in \alpha_iS_i. 
\label{noncausalallocation}
\end{equation}
This is a factorization of $e$. It can be done if each element of the vector $e$ is known apriori at the beginning of the delivery window. However, the uncertain signal $e$ is revealed {\bf causally}. At each time $t$, the operator must irreversibly commit to a resource allocation of the sample $e^t$ without knowing future values.  We now argue that this is not always possible under the solution to (\ref{resourceprocurement}) prescribed in Theorem  \ref{noncausalresourceprocurement}.  

Consider two resources over three time periods, i.e., $N=2$ and $T=3$. Let the demand signal set $E=\text{conv}\{(0,0,0),(1,1,-2),(1,1,4)\}$.  The resources are batteries.  Battery $i$ has a capacity constraint $C_i$ and a maximum charge/discharge rate $r_i$. Unit battery $i$ has a price $\pi_i$. Let  
\begin{align*}
\begin{cases}
&C_1=C_2=3 \\
&r_1=3, r_2=1 \\
&\pi_1=3, \pi_2=1.
\end{cases}
\end{align*}
Assume all batteries are fully discharged at time $0$. Rhen, the feasible energy output $s_1\in S_1$ of the first battery satisfies the  linear constraints:
\begin{align}
\label{constraint4battery1}
\begin{cases}
0\leq s_1^1\leq 3; \\
0\leq s_1^1+s_1^2 \leq 3; \\
0\leq s_1^1+s_1^2+s_1^3 \leq 3, \\
-3\leq s_1^t\leq 3, \quad \forall t=1,2,3.
\end{cases}
\end{align}
The feasible energy output $s_2\in S_2$ of the second battery satisfies the  linear constraints:
\begin{align}
\label{constraint4battery2}
\begin{cases}
0\leq s_2^1\leq 3; \\
0\leq s_2^1+s_2^2 \leq 3; \\
0\leq s_2^1+s_2^2+s_2^3 \leq 3, \\
-1\leq s_2^t\leq 1, \quad \forall t=1,2,3.
\end{cases}
\end{align}
From (\ref{constraint4battery1}) and (\ref{constraint4battery2}), the controllable resource sets $S_i$ are bounded polytopes, and clearly, $0\in \text{int }(S_i)$.
The resource procurement problem without causality constraint is:
\begin{equation}
\label{resourceprocurementexmaple}
\hspace{-1cm} \min_{\alpha_1,\alpha_2} 3\alpha_1+\alpha_2
\end{equation}
\begin{subnumcases}{\label{constraint4non_causalexample}}
E \subseteq \alpha_1 S_1 \oplus \alpha_2 S_2  \label{polytopecontainexample} \\
\alpha_1\geq 0,  \alpha_2\geq 0,
\end{subnumcases}

We show in the Appendix that:
\begin{proposition}
\label{solutiontoexample}
The optimal resource procurement problem without causality constraint (\ref{resourceprocurementexmaple}) has a unique solution  $\alpha_1^*=\alpha_2^*=1$. This optimal asset mix is insufficient to causally cover the demand set $E$. 
\end{proposition}
This motivates us to incorporate causality constraints explicitly in  the resource procurement problem.

\subsection{Optimal Resource Procurement under Causality}

We require the following definitions:

\begin{definition}
A map $\phi: \mathbb{R}^T \rightarrow \mathbb{R}^T: (u^1, u^2, \ldots, u^T)\rightarrow (y^1, y^2, \ldots, y^T)$ is called causal if and only if:
\begin{align*}
 y^1 &= \text{function of } u^1, \\
 y^2 &= \text{function of } (u^1, u^2), \\
 \vdots \\
 y^T &= \text{function of } (u^1, \dots, u^T).
\end{align*}
\end{definition}

\begin{definition}
An allocation policy $\gamma=(\phi_1, \ldots, \phi_N)$ is a collection of maps $\phi_i : E \rightarrow \alpha_i S_i $ such that
\begin{align*}
\sum_i \phi_i(e) = e.
\end{align*}
\end{definition}
In other words, $\gamma$ allocates the uncertain demand signal $e$ to each of the procured resources $\alpha_i S_i$. The sum of these allocations  covers $e$. The policy $\gamma$ can be regarded as a factorization of the identify map $I$ as $\sum_{i\in [N]}\phi_i=I$. 

\begin{definition}
The allocation policy $\gamma$ is said to be causal if and only if its component maps $\phi_i$ are causal. 
\end{definition}

Let $\Gamma$ denote the set of all causal allocation policies. The optimal resource procurement problem under causal allocation can be cast as:

\begin{equation}
\label{resourceprocurement_causality}
\hspace{-1cm} J^{**}=\min_{\alpha_1,\ldots,\alpha_N, \gamma(\cdot)} \sum_{i=1}^{N} \alpha_i p_i
\end{equation}
\begin{subnumcases}
{\label{constraint4causal_formualtion}}
\phi_i(e)\in \alpha_i S_i \\
\sum_{i\in [N]} \phi_i(e)=e \label{polytopecontain_causal} \\
\gamma=(\phi_1,\ldots,\phi_N)\in \Gamma, \alpha_i\geq 0, \label{causal_policy}\\
\forall e\in E,  \label{robustness} 
\end{subnumcases}
where (\ref{polytopecontain_causal}) dictates that $\gamma$ is a factorization of demand signal $e$, (\ref{causal_policy}) restricts the allocation policy to be causal,  and (\ref{robustness}) requires that all signals in $E$ are covered.  This is a joint optimization over the asset mix $\alpha$ and the causal allocation policy $\gamma$.

\subsection{Price of Causality}
The resource procurement problem (\ref{resourceprocurement_causality}) reduces to (\ref{resourceprocurement}) if $\gamma(\cdot)$ is permitted to be non-causal. Therefore, the constraint (\ref{constraint4causal_formualtion}) is more restrictive than (\ref{constraint4non_causal}), and $J^{**}\geq J^*$.  Requiring that the allocation policy be causal inflates the optimal resource procurement cost from $J^*$ to $J^{**}$. We define the price of causality (PoC) as:
\begin{equation}
\label{definitionofPoC}
\text{PoC}=\dfrac{J^{**}}{J^*}.
\end{equation}
This captures the cost premium necessary for causal allocation. It is clear that $PoC \geq 1$. A large $PoC$ indicates that there is a large cost premium associated with causal allocation of the uncertain demand signals. Perhaps, a forward market for procuring reserves is not a suitable mechanism in such a case. On the other hand, $PoC \approx 1$ suggests that there is a minimal additional cost for not knowing the future values of the uncertain demand signal.

\section{Special Cases}
Optimal resource procurement under causality (\ref{resourceprocurement_causality})   is  an adjustable multi-stage robust optimization \cite{ben2009robust}, which is well-known to be challenging.  Instead of solving this problem, we compute upper bounds on $J^{**}$ by restricting to a class of allocation policies. Separately, we compute the exact price of causality in some special cases.

\subsection{Proportional Allocation Policy}
The simplest class of causal allocation policy are  proportional strategies. Here, the  uncertain demand signals $e$ are allocated to procured resources according to a fixed proportion $\beta=(\beta_1,\ldots,\beta_N)$ as
\begin{equation}
\phi_i(e)=\beta_i e, \text{where } \sum_{i=1}^N \beta_i=1, \beta_i\geq 0. 
\end{equation}
The allocation $\phi_i$ is clearly causal by construction. Under this policy, the resource procurement problem (\ref{resourceprocurement_causality}) reduces to:
\begin{equation}
\label{causality4proportaional}
\hspace{-1cm} \min_{\alpha_1,\ldots,\alpha_N, \beta} \sum_{i=1}^{N} \alpha_i \pi_i
\end{equation}
\begin{subnumcases}
{\label{constraint4causal_formualtion_prop}}
\beta_i E \subseteq  \alpha_i S_i, \quad \forall i\in [N] \label{polytopecontain_causal_prop} \\
\sum_{i=1}^N \beta_i=1, \\
\alpha_i\geq0,  \quad \forall i\in [N].
\end{subnumcases}
The optimal value of (\ref{causality4proportaional}) offers an upper bound for $J^{**}$. Under proportional allocation, it happens that the operator procures a single controllable resources.  More precisely:
\begin{proposition}
\label{meritorderproposition}
The optimal solution to (\ref{causality4proportaional}), $(\alpha_1^*,\ldots,\alpha_N^*, \beta^*)$, satisfies  $\alpha_j^*\neq 0$ for only one index $j$.  
\end{proposition}
The proof is deferred to the Appendix. Proposition \ref{meritorderproposition} suggests that under proportional allocation, each resource has a ``virtual price" $k_i\pi_i$, where $k_i$ captures the shape of the polytope $S_i$. The virtual prices determine a ``merit order" of the resources: at the optimal solution, only the ``cheapest" resource is selected.

\begin{remark}
A natural extension is to consider time-varying proportional allocation: 
\begin{equation}
\phi_i(e)=(\beta_i^1e^1, \ldots, \beta_i^Te^T)
\end{equation}
This offers a tighter upper bound on $J^{**}$. Under time-varying proportional allocation, resources can not be merit-ordered, and Proposition \ref{meritorderproposition} no longer holds. 
\end{remark}

\subsection{Causal-Affine Policies}
Consider a more general class of  policies as follows:
\begin{definition}
\label{linearpolicy}
The allocation policy $\gamma(\cdot)$ is called causal-affine if for any $i\in [N]$, there exist lower-triangular matrices $F_i\in \mathbb{R}^{T\times T}$ and vectors $D_i\in \mathbb{R}^T$ such that:
\begin{align}
\label{affinepolicydef}
\begin{cases}
 \phi_i(e)=F_i e+D_i, \\
 \sum_{i=1}^N F_i=I, \quad \sum_{i=1}^N D_i=0, 
\end{cases}
\end{align}
where $I$ is the identity matrix. 
\end{definition}
Such policies are causal by virtue of the lower-triangular constraints on $F_i$ (as for linear time-varying systems).
Under causal-affine  policy, the resource procurement problem  can be solved with a linear program:
\begin{theorem}
\label{linearupperbound}
The optimal resource procurement problem (\ref{resourceprocurement_causality}) restricted to causal-affine policies is equivalent to:
\begin{equation}
\label{resourceprocurement_causal_linear}
\hspace{-4.5cm} \min \sum_{i=1}^{N} \alpha_i \pi_i
\end{equation}
\begin{subnumcases}{\label{constraint4causal}}
v_k=\sum_{i=1}^N \left(F_iv_k+D_i\right), \quad \forall k\in [K], \label{conservationofall_linear}\\
\sum_{i=1}^N F_i=I, \quad \sum_{i=1}^N D_i=0,   \\
A_i\phi_i(v_k)\leq \alpha_iB_i,  \quad \forall i,k, \label{feasibleset_i}\\
\alpha_i\geq 0,  \quad \forall i\in [N],
\end{subnumcases}
The decision variables are $\alpha_1, \ldots, \alpha_N, F_i$ and $D_i$ for all $i,t$.
\end{theorem}

The proof can be found in the Appendix. The solution to (\ref{resourceprocurement_causal_linear}) offers an improved  upper bound on $J^{**}$.

\subsection{Identical or Static Resources}
In some special cases, the price of causality can be determined exactly. One interesting case is when  the resource sets $S_i$ are identical up to a scale factor:
\begin{proposition}
\label{poais1_3}
Assume there exists a $\delta_i\in \mathbb{R}$, such that $\delta_iS_i=\delta_jS_j$ for  $\forall i,j\in [N]$.    Then, $\text{PoC}=1$.
\end{proposition}

A second interesting case is when the resources have no dynamic constraints. This happens when the resource sets $S_i$ are hyper-rectangles:
\begin{equation}
S_i=\{s_i\in \mathbb{R}^T |
\underline{s}_i^t \leq s_i^t  \leq \bar{s}_i^t, t\in [T] \}.
\end{equation}
This is because there is no constraint coupling between sample values $s_i^t$.  We have the following:
\begin{proposition}
\label{poais1_2}
For $\forall i\in [N]$, let $S_i$ be a hyper-rectangle: $S_i=\{s_i\in \mathbb{R}^T |
\underline{s}_i^t \leq s_i^t  \leq \bar{s}_i^t, t\in [T] \}$.  Then,  $\text{PoC}=1$.
\end{proposition}
The proofs of Proposition \ref{poais1_3} and Proposition \ref{poais1_2} are deferred to the Appendix. 

\begin{remark}
Proposition \ref{poais1_3}-\ref{poais1_2} suggest that  $PoC>1$ is because of the dynamics constraints and the diversity  of the controllable resources.  The price of causality is greater than 1 only if both factors are present. 
\end{remark}

\subsection{Special Uncertain Signals}
Proposition \ref{poais1_2} argues that $PoC > 1$ due to the diversity and dynamics of the resources. A closely related question is whether the dynamics of the uncertain demand signals contribute materially to the price of causality. We have the following:
\begin{proposition}
\label{poais1}
Assume that $E$ is the hyper-rectangle
\begin{equation*}
E=\{e\in \mathbb{R}^T |
\underline{e}^t \leq e^t  \leq \bar{e}^t, \forall t\in [T]\}.
\end{equation*}Then, it is possible that $\text{PoC}>1$.
\end{proposition}

Proposition \ref{poais1} suggests that the dynamics of the uncertain demand signals is not the key driver for the price of causality. If the demand signals are temporally correlated, one could forecast future signal values based on the current sample. This can make the price of causality small. Consider for instance, the intermittent power output of a wind farm which causes the need to procure balancing power. Day ahead forecasting of wind generation can have large errors ($\pm 30\%$,  see \cite{giebel2003state}). On the other hand, hour-ahead forecasts can be fairly accurate ($\pm 5\%$, see \cite{giebel2003state}]). Therefore, the future signal values are known with high accuracy after the first uncertain signal sample $e^1$ is revealed. This is close to the non-causal case where we know all future signal values in advance. Then, we conjecture that $PoC\approx 1$.


%
%
%

\subsection{Electricity Storage}
We can obtain the exact price of causality for an important class of problems where the controllable resources are electricity storage devices. 

Consider a collection of $N$ batteries.  Battery $i$ has a capacity constraint $C_i$ and a maximum charge/discharge rate $r_i$. 
Denote the initial state of charge for the $i$th battery (as a percentage of capacity) as $\theta_i$. Then, $S_i$ contains all $s_i\in \mathbb{R}^T$ that satisfy the constraints:
\begin{numcases}{}
-r_i\leq s_i^t \leq r_i, \quad \forall t\in [T],  \label{rateconstraint}\\
0\leq \theta_iC_i+\sum_{k=1}^t s_i^k\leq C_i, \quad \forall t\in [T], i\in [N], \label{capacityconstaint}
\end{numcases}

We  show that $J^{**}$ can be obtained by solving a linear program. 
\begin{theorem}
\label{proposition4exactpoa}
Consider a group of batteries. Assume  their parameters satisfy $\sum_{i=1}^N C_i\leq 2\sum_{i=1}^N r_i $, and that $E=S_1\oplus S_2 \oplus \cdots \oplus S_N$, then there exists $T_0>0$, so that for $\forall T> T_0$,  $J^{**}$ is the optimal value for the following linear program:
\begin{equation}
\label{poaaslinearprogramming}
\hspace{-2cm} \min_{\alpha_1,\ldots,\alpha_N} \sum_{i=1}^{N} \alpha_i \pi_i
\end{equation}
\begin{subnumcases}{\label{constraint4poaaslinearprogramming}}
\sum_{i=1}^N \alpha_i r_i \geq \sum_{i=1}^N r_i, \label{constraint1forpoaaslp}\\
\sum_{i=1}^{N} \alpha_i  \min (2r_i,C_i)\geq\sum_{i=1}^N C_i, 
 \label{constraint2forpoaaslp}
\end{subnumcases}
 where $\min (a,b)$ is the smallest element of $\{ a, b \}$.
\end{theorem}
The proof of Theorem \ref{proposition4exactpoa} can be found in the Appendix. 
Under the assumption of Theorem \ref{proposition4exactpoa}, we can solve  (\ref{resourceprocurement_linear}) and (\ref{poaaslinearprogramming}) to derive $J^*$ and $J^{**}$, respectively. This provides the exact price of causality.

\begin{remark}
 Theorem \ref{proposition4exactpoa} relies on a crucial condition: $E=S_1 \oplus \cdots \oplus S_N$. We emphasize that this condition is not too restrictive if there are large number of sufficiently diverse resources. In this case, for  given  $S_i$ and $\pi_i$, we can first find $(\eta_1,\ldots,\eta_N)$ such that $E$ is sufficiently close to $\eta_1S_1 \oplus \cdots \oplus \eta_N S_N$. We define the unit resource as $\eta_iS_i$ and the unit price as $\eta_ip_i$. It is easy to verify that this is equivalent to (\ref{poaaslinearprogramming}).
\end{remark}


\section{Fair Cost Allocation}
The optimal resource procurement cost $J^{**}$ must be allocated to the uncontrollable resources that collectively create the uncertain demand signal $e$. We study how to allocate this cost fairly. 

Consider a group of $L$ uncontrollable resources. Resource $i$ contributes $d_i\in \mathbb{R}^T$ to the collective demand signal, so  $e=\sum_{i=1}^N d_i$. 
A cost allocation mechanism is a collection of maps
\begin{equation}
\pi_i: \mathbb{R}^T \times \mathbb{R}^T \rightarrow \mathbb{R}: (d_i, e) \rightarrow J_i
\end{equation}  
This mechanism maps  individual demands $d_i$  and the aggregate demand $e$ to the cost $J_i$ allocated to resource $i$, i.e., $J_i = \pi_i (d_i, e)$. A just and reasonable cost allocation should satisfy equity, budget balance and fairness. More formally, we have
\begin{axiom}[Equity]
The cost allocation $\pi_i(d_i,e)$ is equal if players with the same contribution  have the same cost allocation, i.e., if $d_i=d_j$, then $\pi_i(d_i,e)=\pi_j(d_j,e)$.
\label{axiom1}
\end{axiom}

\begin{axiom}[Budget Balance]
The cost allocation $\pi_i(d_i,e)$ is budget balanced,   i.e., $\sum_{i=1}^L \pi_i(d_i,e)=J^{**}$.
\label{axiom2}
\end{axiom}

\begin{axiom}[Penalty for Cost Causation]
Those who contribute to the uncertain signals should pay for it,  i.e., if $d_i\cdot e>0$, $\pi_i(d_i,e)>0$.
\label{axiom3}
\end{axiom}

\begin{axiom}[Reward for Cost Mitigation]
Those who mitigate the uncertain signals should be rewarded,  i.e., if $d_i\cdot e<0$, $\pi_i(d_i,e)<0$.
\label{axiom4}
\end{axiom}

We refer to these Axioms collectively as the cost causation principle \cite{chakraborty2017cost}.  
We have the following:
\begin{proposition}
\label{payment}
The cost allocation mechanism 
\begin{equation}
\pi_i(d_i,e)= \dfrac{d_i^T e}{||e||^2} J^{**}
\end{equation}  
satisfies the cost causation principle. 
\end{proposition}

The proof of Proposition \ref{payment} easily follows from the definition of the axioms, and is thus omitted.

\begin{figure}[t!]
\centering
%
%
\definecolor{mycolor1}{rgb}{0.60000,0.20000,0.00000}%
\begin{tikzpicture}

\begin{axis}[%
width=2.4in,
height=1in,
at={(3.216in,1.601in)},
scale only axis,
xmin=0,
xmax=10,
xlabel={$\text{ratio }\kappa$},
ymin=0,
ymax=7,
ylabel={Procurement Cost},
axis background/.style={fill=white},
legend style={at={(0,0.6)}, anchor=south west, legend cell align=left, align=left, draw=black}
]
\addplot [color=mycolor1, line width=1.5pt]
  table[row sep=crcr]{%
0	5.06616970596951e-11\\
0.5	1.00000000102034\\
1	2.00000000000995\\
1.5	2.5000000000031\\
2	3.00000000000455\\
2.5	3.50000000007878\\
3	4.00000000000205\\
3.5	4.00000000109316\\
4	4.00000000000944\\
4.5	4.00000000008544\\
5	4.00000000019111\\
5.5	4.00000000009149\\
6	4.00000000768966\\
6.5	4.00000000301185\\
7	4.00000000120281\\
7.5	4.00000000013611\\
8	4.00000000012676\\
8.5	4.00000000012204\\
9	4.00000000011426\\
9.5	4.00000000009626\\
10	4.00000000007162\\
};
\addlegendentry{J*}

\addplot [color=black, dashdotted, line width=1.5pt]
  table[row sep=crcr]{%
0	1.27560814172335e-13\\
0.5	1.0000000000325\\
1	2.00000000134843\\
1.5	3.00000000146667\\
2	4.00000000543158\\
2.5	4.00000000000042\\
3	4.00000000000014\\
3.5	4.00000000000002\\
4	4.00000002534246\\
4.5	4.00000002315145\\
5	4.00000001868696\\
5.5	4.00000001498373\\
6	4.00000001233985\\
6.5	4.00000001028199\\
7	4.00000000841584\\
7.5	4.00000000671571\\
8	4.00000000519044\\
8.5	4.00000000386544\\
9	4.00000000276357\\
9.5	4.00000000202787\\
10	4.00000000000294\\
};
\addlegendentry{J**}

\end{axis}
\end{tikzpicture}%
\caption{The  procurement costs under different price ratio $\kappa$.}
\label{figure1}
\end{figure}
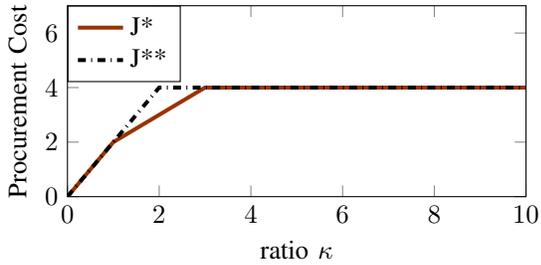

\begin{figure}[t!]
\centering
%
%
\definecolor{mycolor1}{rgb}{0.60000,0.20000,0.00000}%
\begin{tikzpicture}

\begin{axis}[%
width=2.4in,
height=1in,
at={(3.216in,1.601in)},
scale only axis,
xmin=0,
xmax=10,
xlabel={$\text{ratio }\kappa$},
ymin=0.999999999012165,
ymax=1.4,
ylabel={price of causality},
axis background/.style={fill=white}
]
\addplot [color=mycolor1, line width=1.5pt, forget plot]
  table[row sep=crcr]{%
0	1\\
0.5	0.999999999012165\\
1	1.00000000066924\\
1.5	1.20000000058518\\
2	1.33333333514184\\
2.5	1.14285714283154\\
3	0.999999999999523\\
3.5	0.999999999726715\\
4	1.00000000633326\\
4.5	1.0000000057665\\
5	1.00000000462396\\
5.5	1.00000000372306\\
6	1.00000000116255\\
6.5	1.00000000181754\\
7	1.00000000180326\\
7.5	1.0000000016449\\
8	1.00000000126592\\
8.5	1.00000000093585\\
9	1.00000000066233\\
9.5	1.0000000004829\\
10	0.99999999998283\\
};
\end{axis}
\end{tikzpicture}%
\caption{The price of causality under different price ratio $\kappa$.}
\label{figure2}
\end{figure}
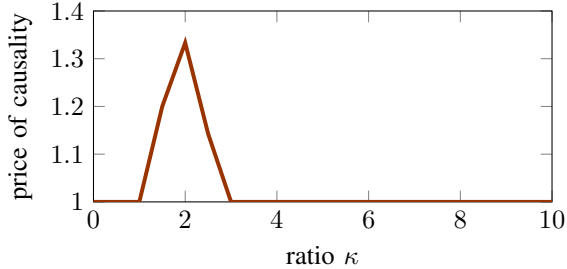

\section{Simulation Studies}

\subsection{Electricity Storage}
We consider two batteries that cover a signal over a $T=3$ period delivery window. The capacity of the batteries are $C_1=1$ and $C_2=3$, and the maximum charge/discharge rates are $r_1=r_2=1$. Assume the batteries fully discharged at time $0$. Then $S_1$ is the set of $s_1\in \mathbb{R}^3$ such that:
\begin{align*}
\begin{cases}
-1\leq s_1^t \leq 1, \quad \forall t=1,2,3. \\
0\leq s_1^1+s_1^2+s_1^3\leq 1.
\end{cases}
\end{align*}
and $S_2$  is the set of $s_2\in \mathbb{R}^3$ such that:
\begin{align*}
\begin{cases}
-1\leq s_2^t \leq 1, \quad \forall t=1,2,3. \\
0\leq s_2^1+s_2^2+s_2^3\leq 3. 
\end{cases}
\end{align*}
The uncertain demand signal $e$ is contained in $E = S_1 \oplus S_2$.. All assumptions in Theorem \ref{proposition4exactpoa} are satisfied and we can compute the exact price of causality.

We explore the influence of the price vector $(\pi_1,\pi_2)$ on the price of causality. Define $\kappa=\pi_2/\pi_1$. In this simulation, we fix $\pi_1=1$ and vary $\kappa$ from $0$ to $10$ in increments of $0.1$. The optimal procurement costs $J^*$ and $J^{**}$ are shown in Figure \ref{figure1}, and the price of causality as a function of $\kappa$ is shown in Figure \ref{figure2}. In the extremes, $\kappa<1$ or $\kappa>3$, one battery is too expensive, and only the cheaper battery is procured, i.e., either $\alpha_1=0$  or $\alpha_2=0$. In this case, the allocation problem is trivial, and the price of causality is 1. In the intermediate range, $1\leq \kappa \leq 3$, both batteries are procured, and causal revelation of the demand signal influences the optimal cost. The maximum price of causality can be as large as $1.33$.

 \begin{figure*}[bt]%
\begin{minipage}[b]{0.49\linewidth}
\centering
\input{figure3}
\caption{The procurement costs under different price ratio $\kappa$.}
\label{figure3}
\end{minipage}
\begin{minipage}[b]{0.02\linewidth}
\hfill
\end{minipage}
\begin{minipage}[b]{0.49\linewidth}
\centering
%
%
\definecolor{mycolor1}{rgb}{0.60000,0.20000,0.00000}%
\begin{tikzpicture}

\begin{axis}[%
width=2.4in,
height=1in,
at={(3.216in,1.601in)},
scale only axis,
unbounded coords=jump,
xmin=0,
xmax=4,
xlabel style={font=\color{white!15!black}},
xlabel={$\kappa$},
ymin=1,
ymax=1.04,
ylabel style={font=\color{white!15!black}},
ylabel={Price of Causality},
axis background/.style={fill=white}
]
\addplot [color=mycolor1, line width=1.5pt, forget plot]
  table[row sep=crcr]{%
0	nan\\
0.02	1\\
0.04	1\\
0.06	1\\
0.08	1\\
0.1	1\\
0.12	1\\
0.14	1\\
0.16	1\\
0.18	1\\
0.2	1\\
0.22	1\\
0.24	1\\
0.26	1\\
0.28	1\\
0.3	1\\
0.32	1\\
0.34	1\\
0.36	1\\
0.38	1\\
0.4	1\\
0.42	1\\
0.44	1\\
0.46	1\\
0.48	1\\
0.5	1\\
0.52	1\\
0.54	1\\
0.56	1\\
0.58	1\\
0.6	1\\
0.62	1\\
0.64	1\\
0.66	1\\
0.68	1\\
0.7	1\\
0.72	1\\
0.74	1\\
0.76	1\\
0.78	1\\
0.8	1\\
0.82	1\\
0.84	1\\
0.86	1\\
0.88	1\\
0.9	1\\
0.92	1\\
0.94	1\\
0.96	1\\
0.98	1\\
1	1\\
1.02	1.00114494997663\\
1.04	1.00225105138419\\
1.06	1.0033202484575\\
1.08	1.00435435782394\\
1.1	1.0053550788037\\
1.12	1.006324002728\\
1.14	1.00726262138318\\
1.16	1.00817233467458\\
1.18	1.00905445759305\\
1.2	1.00991022655714\\
1.22	1.01074080519521\\
1.24	1.01154728962441\\
1.26	1.01233071327665\\
1.28	1.01309205131662\\
1.3	1.01383222469119\\
1.32	1.01455210384583\\
1.34	1.01525251213943\\
1.36	1.01593422898576\\
1.38	1.01659799274667\\
1.4	1.01724450339981\\
1.42	1.01787442500079\\
1.44	1.01848838795849\\
1.46	1.01908699113932\\
1.48	1.01967080381568\\
1.5	1.02024036747171\\
1.52	1.02079619747837\\
1.54	1.02133878464887\\
1.56	1.02186859668414\\
1.58	1.02238607951751\\
1.6	1.02289165856644\\
1.62	1.023385739899\\
1.64	1.02386871132156\\
1.66	1.02434094339399\\
1.68	1.0248027903778\\
1.7	1.02525459112249\\
1.72	1.02569666989469\\
1.74	1.02612933715423\\
1.76	1.02655289028125\\
1.78	1.02696761425789\\
1.8	1.02737378230767\\
1.82	1.02777165649579\\
1.84	1.02816148829299\\
1.86	1.02854351910558\\
1.88	1.02891798077391\\
1.9	1.02928509604163\\
1.92	1.02964507899742\\
1.94	1.02999813549141\\
1.96	1.03034446352761\\
1.98	1.03068425363427\\
2	1.03101768921338\\
2.02	1.03134494687076\\
2.04	1.03166619672794\\
2.06	1.0319816027171\\
2.08	1.03229132285999\\
2.1	1.03259550953189\\
2.12	1.03289430971162\\
2.14	1.03318786521833\\
2.16	1.03347631293592\\
2.18	1.0337597850259\\
2.2	1.03403840912929\\
2.22	1.03431230855822\\
2.24	1.03458160247798\\
2.26	1.03484640607979\\
2.28	1.03510683074516\\
2.3	1.03536298420209\\
2.32	1.03561497067363\\
2.34	1.03586289101928\\
2.36	1.03610684286961\\
2.38	1.03634692075441\\
2.4	1.03658321622487\\
2.42	1.03681581796995\\
2.44	1.03704481192738\\
2.46	1.03727028138952\\
2.48	1.03749230710435\\
2.5	1.03771096737193\\
2.52	1.03792633813643\\
2.54	1.0381384930741\\
2.56	1.03834750367727\\
2.58	1.03855343933471\\
2.6	1.0387563674084\\
2.62	1.03895635330702\\
2.64	1.03915346055624\\
2.66	1.03920166172148\\
2.68	1.03923799244478\\
2.7	1.0392738112668\\
2.72	1.0393091289309\\
2.74	1.03934395588192\\
2.76	1.03350789965783\\
2.78	1.02642688992626\\
2.8	1.01944225701108\\
2.82	1.01255204659207\\
2.84	1.00575435683361\\
2.86	1\\
2.88	1\\
2.9	1\\
2.92	1\\
2.94	1\\
2.96	1\\
2.98	1\\
3	1\\
3.02	1\\
3.04	1\\
3.06	1\\
3.08	1\\
3.1	1\\
3.12	1\\
3.14	1\\
3.16	1\\
3.18	1\\
3.2	1\\
3.22	1\\
3.24	1\\
3.26	1\\
3.28	1\\
3.3	1\\
3.32	1\\
3.34	1\\
3.36	1\\
3.38	1\\
3.4	1\\
3.42	1\\
3.44	1\\
3.46	1\\
3.48	1\\
3.5	1\\
3.52	1\\
3.54	1\\
3.56	1\\
3.58	1\\
3.6	1\\
3.62	1\\
3.64	1\\
3.66	1\\
3.68	1\\
3.7	1\\
3.72	1\\
3.74	1\\
3.76	1\\
3.78	1\\
3.8	1\\
3.82	1\\
3.84	1\\
3.86	1\\
3.88	1\\
3.9	1\\
3.92	1\\
3.94	1\\
3.96	1\\
3.98	1\\
4	1\\
};
\end{axis}
\end{tikzpicture}%
\caption{The price of causality under different price ratio $\kappa$.}
\label{figure4}
\end{minipage}
\end{figure*}

\subsection{Energy Reserves}
We study the energy reserve procurement problem introduced in Section \ref{energyprocurementsection}. The operator procures reserves from a forward capacity market to balance supply and demand. The imbalance signal is revealed every 5 minutes during a 30 minute delivery window, i.e., $T=6$.

There are two types of balancing resources: a slow diesel generator and a gas turbine generator. Each generator has a capacity constraint and a ramp rate constraint. We set the generator capacities at $10MW$. Typically, the slow diesel generator has a ramp rate constraint of $7\%/\text{min}$ of its capacity, while the fast turbine generator can ramp up and down its full capacity in 5 minutes \cite{gonzalez2017review}. The nominal operation points for both generators is $5MW$. Let $s_i^t$ be the power deviations from this nominal value. The constraints for the slow generator are:
\begin{align*}
\begin{cases}
-5MW\leq s_1^t \leq 5MW, \quad \forall t=1,\ldots,6. \\
-3.5MW\leq |s_1^t-s_1^{t-1}|\leq 3.5MW.
\end{cases}
\end{align*}
The constraints for the fast generator are:
\begin{equation*}
-5MW\leq s_2^t \leq 5MW, \quad \forall t=1,\ldots,6. 
\end{equation*}
The sets $S_1$ and $S_2$ can be defined accordingly.

%
%

We use frequency regulation signals from the PJM market \cite{regulationdata} to construct $E$. The regulation signal is normalized between $-1$ and $1$. It is revealed every two seconds and it indicates the power imbalances of the grid when multiplied by the total reserve capacity. We use historic RegA data from the year 2017, and compute the average power imbalance in every 5 minutes. We divide the entire yearly 5-minute average trajectory into $17520$ segments. Each segment corresponds to a 30 minute interval, denoted $e_i\in R^6$. We view each segment as a sample of the imbalance signal, and our first goal is to construct $E$ from these samples so that the future imbalance signals lie in $E$ with very high probability.  We denote the entire data set as $\mathcal{D}=\{e_1,\ldots,e_{\text{17520}}\}$, and  partition $\mathcal{D}$  into a training set $\mathcal{D}_1 = \{e_1,\ldots,e_{\text{10000}}\}$ and a validation set $\mathcal{D}_2 = \{e_{\text{10001}},\ldots,e_{\text{17520}}\}$. We use $\mathcal{D}_1$ to construct $E$ and $\mathcal{D}_2$ to test the  model. A naive approach is to define $E = \text{conv }(\mathcal{D}_1)$. However, this approach has poor performance: only $76\%$ of the data from $\mathcal{D}_2$ is contained in $E$. We therefore inflate $E$ by a scaling factor $\delta$. Surprisingly, the coverage ratio increases to $93\%$ at $\delta=1.01$. This indicates that enlarging the set $E$ by $1\%$ can improve the performance significantly. Figure \ref{coverage}  shows the coverage ratio as a function of  $\delta$ allowing us to choose $\delta$ for a desired level of coverage. We can choose $\delta$ according to desired level of coverage based on Figure \ref{coverage}.

\begin{figure}[h!]%
\centering
%
%
\definecolor{mycolor1}{rgb}{0.60000,0.20000,0.00000}%
\begin{tikzpicture}

\begin{axis}[%
width=2.4in,
height=1in,
at={(3.216in,1.601in)},
scale only axis,
xmin=0,
xmax=0.1,
xtick={0,0.02,0.04,0.06,0.08,0.1},
xticklabels={{0},{2\%},{4\%},{6\%},{8\%},{10\%}},
xlabel style={font=\color{white!15!black}},
xlabel={Percentage of Inflation},
ymin=0.7,
ymax=1,
ytick={0.7,0.8,0.9,1},
yticklabels={{70\%},{80\%},{90\%},{1}},
ylabel style={font=\color{white!15!black}},
ylabel={Coverage Ratio},
axis background/.style={fill=white}
]
\addplot [color=mycolor1, line width=1.5pt, forget plot]
  table[row sep=crcr]{%
0	0.7501312335958\\
0.01	0.928215223097113\\
0.02	0.947637795275591\\
0.03	0.957742782152231\\
0.04	0.966929133858268\\
0.05	0.973622047244095\\
0.0600000000000001	0.979790026246719\\
0.0700000000000001	0.984514435695538\\
0.0800000000000001	0.987270341207349\\
0.0900000000000001	0.991732283464567\\
0.1	0.992782152230971\\
};
\end{axis}
\end{tikzpicture}%
\caption{The coverage ratio as a function of percentage of increase.}
\label{coverage}
\end{figure}
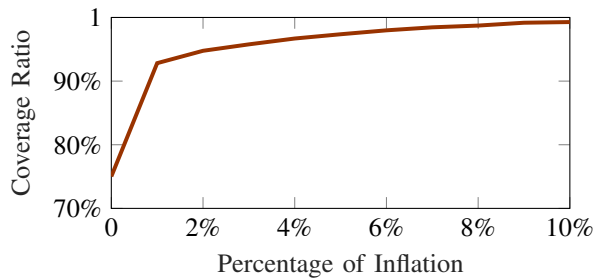

We explore the influence of unit prices affect the price of causality. Let $\pi_1$ and $\pi_2$ be the unit price of the diesel generator and turbine generator, respectively. 
Define $\kappa=\pi_2/\pi_1$. We fix $\pi_1=1$ and vary $\kappa$ from $0$ to $4$ in increments of $0.01$. The optimal procurement cost $J^*$ and our upper bound $\bar{J}$ are shown in Figure \ref{figure3}, and the price of causality is shown in \ref{figure4}. At the extremes, $\kappa<1$ or $\kappa>2.86$, one of the generators is too expensive, and only the cheaper resource is procured. The price of causality in these cases is 1. In the intermediate case, $1 < \kappa < 2.86$, both generators are procured and the price of causality can be greater then 1. From Figure \ref{figure4}, the upper bound on the price of causality is $1.04$. We note that this number may be economically significant: it is estimated that a $1\%$ increase of reserve requirements costs $50$ million dollars per year in California alone.

\section{conclusion}
This paper considered the problem of procuring diverse resources ex-ante to cover an uncertain demand signal. We considered two examples: reserve procurement in electricity market and instance procurement for cloud computing. Through these examples, we have shown that causality induces an additional procurement cost. We formulated the optimal resource procurement as a set containment linear program. An upper bound on the price of causality is obtained by restricting allocation policies to be affine. The exact price of causality is derived in some special cases, and all computations are based on linear programming. A cost-allocation mechanism is proposed. It satisfies the equity, budget balance, and fairness. Simulation results show interesting dependence of the price of causality on resource prices. Future research includes deriving lower bound on the price of causality, analyzing responsive loads, and exploring endogenous price discovery. We have assumed that resource prices $\pi_i$ are set by the seller. An intriguing possibility is to study general forward markets for trading  diverse resources modeled convex sets.

\bibliographystyle{unsrt}
\bibliography{resourceprocurement}

\begin{thebibliography}{10}

\bibitem{shetty2018optimal}
A.~Shetty, S.~Li, K.~Poolla, and P.~Varaiya.
\newblock Optimal energy reserve procurement.
\newblock In {\em European Control Conference}, 2018.

\bibitem{borodin2005online}
A.~Borodin and R.~El-Yaniv.
\newblock {\em Online computation and competitive analysis}.
\newblock cambridge university press, 2005.

\bibitem{buchbinder2009design}
N.~Buchbinder and J.~S. Naor.
\newblock The design of competitive online algorithms via a primal--dual
  approach.
\newblock {\em Foundations and Trends{\textregistered} in Theoretical Computer
  Science}, 3(2--3):93--263, 2009.

\bibitem{brown2014information}
D.~B. Brown and J.~E. Smith.
\newblock Information relaxations, duality, and convex stochastic dynamic
  programs.
\newblock {\em Operations Research}, 62(6):1394--1415, 2014.

\bibitem{mao2016optimal}
Z.~Mao, C.~E. Koksal, and N.~B. Shroff.
\newblock Optimal online scheduling with arbitrary hard deadlines in multihop
  communication networks.
\newblock {\em IEEE/ACM Transactions on Networking}, 24(1):177--189, 2016.

\bibitem{hao2017online}
F.~Hao, M.~Kodialam, T.~Lakshman, and S.~Mukherjee.
\newblock Online allocation of virtual machines in a distributed cloud.
\newblock {\em IEEE/ACM Transactions on Networking}, 25(1):238--249, 2017.

\bibitem{bhalgat2012online}
A.~Bhalgat, J.~Feldman, and V.~Mirrokni.
\newblock Online allocation of display ads with smooth delivery.
\newblock In {\em Proceedings of the 18th ACM SIGKDD international conference
  on Knowledge discovery and data mining}, pages 1213--1221. ACM, 2012.

\bibitem{lin2012online}
M.~Lin, Z.~Liu, A.~Wierman, and L.~L.~H. Andrew.
\newblock Online algorithms for geographical load balancing.
\newblock In {\em International Green Computing Conference}, pages 1--10. IEEE,
  2012.

\bibitem{dertouzos1989multiprocessor}
M.~L. Dertouzos and A.~K. Mok.
\newblock Multiprocessor online scheduling of hard-real-time tasks.
\newblock {\em IEEE Transactions on software engineering}, 15(12):1497--1506,
  1989.

\bibitem{subramanian2013real}
A.~Subramanian, M.~J. Garcia, D.~S. Callaway, K.~Poolla, and P.~Varaiya.
\newblock Real-time scheduling of distributed resources.
\newblock {\em IEEE Transactions on Smart Grid}, 4(4):2122--2130, 2013.

\bibitem{wenzel2017real}
G.~Wenzel, M.~Negrete-Pincetic, D.~E. Olivares, J.~MacDonald, and D.~S.
  Callaway.
\newblock Real-time charging strategies for an electric vehicle aggregator to
  provide ancillary services.
\newblock {\em IEEE Transactions on Smart Grid}, 2017.

\bibitem{madjidian2017energy}
D.~Madjidian, M.~Roozbehani, and M.~A. Dahleh.
\newblock Energy storage from aggregate deferrable demand: Fundamental
  trade-offs and scheduling policies.
\newblock {\em to appear in IEEE Transactions on Power Systems}, 2017.

\bibitem{nayyar2016duration}
A.~Nayyar, M.~Negrete-Pincetic, K.~Poolla, and P.~Varaiya.
\newblock Duration-differentiated energy services with a continuum of loads.
\newblock {\em IEEE Transactions on Control of Network Systems}, 3(2):182--191,
  2016.

\bibitem{negrete2016rate}
M.~Negrete-Pincetic, A.~Nayyar, K.~Poolla, F.~Salah, and P.~Varaiya.
\newblock Rate-constrained energy services in electricity.
\newblock {\em IEEE Transactions on Smart Grid}, 2016.

\bibitem{hao2015aggregate}
H.~Hao, B.~M. Sanandaji, K.~Poolla, and T.~L. Vincent.
\newblock Aggregate flexibility of thermostatically controlled loads.
\newblock {\em IEEE Transactions on Power Systems}, 30(1):189--198, 2015.

\bibitem{zhao2017geometric}
L.~Zhao, W.~Zhang, H.~Hao, and K.~Kalsi.
\newblock A geometric approach to aggregate flexibility modeling of
  thermostatically controlled loads.
\newblock {\em to appear in IEEE Transactions on Power Systems}, 2017.

\bibitem{garg2011sla}
S.~K. Garg, S.~K. Gopalaiyengar, and R.~Buyya.
\newblock Sla-based resource provisioning for heterogeneous workloads in a
  virtualized cloud datacenter.
\newblock In {\em International conference on Algorithms and architectures for
  parallel processing}, pages 371--384. Springer, 2011.

\bibitem{carrera2008enabling}
D.~Carrera, M.~Steinder, I.~Whalley, J.~Torres, and E.~Ayguad{\'e}.
\newblock Enabling resource sharing between transactional and batch workloads
  using dynamic application placement.
\newblock In {\em Proceedings of the 9th ACM/IFIP/USENIX International
  Conference on Middleware}, pages 203--222. Springer-Verlag New York, Inc.,
  2008.

\bibitem{amazonec2}
{\em Amazon Elastic Cloud Service}.
\newblock [online] https://aws.amazon.com/ecs/.

\bibitem{li2010cloudcmp}
A.~Li, X.~Yang, S.~Kandula, and M.~Zhang.
\newblock Cloudcmp: comparing public cloud providers.
\newblock In {\em Proceedings of the 10th ACM SIGCOMM conference on Internet
  measurement}, pages 1--14. ACM, 2010.

\bibitem{mangasarian2002set}
O.~Mangasarian.
\newblock Set containment characterization.
\newblock {\em Journal of Global Optimization}, 24(4):473--480, 2002.

\bibitem{ben2009robust}
A.~Ben-Tal, L.~El~Ghaoui, and A.~Nemirovski.
\newblock {\em Robust optimization}.
\newblock Princeton University Press, 2009.

\bibitem{giebel2003state}
Gregor Giebel, Georges Kariniotakis, and Richard Brownsword.
\newblock The state-of-the-art in short-term prediction of wind power-from a
  danish perspective.

\bibitem{chakraborty2017cost}
P.~Chakraborty, E.~Baeyens, and P.~P. Khargonekar.
\newblock Cost causation based allocations of costs for market integration of
  renewable energy.
\newblock {\em to appear in IEEE Transactions on Power Systems}, 2017.

\bibitem{gonzalez2017review}
M.~A. Gonzalez-Salazar, T.~Kirsten, and L.~Prchlik.
\newblock Review of the operational flexibility and emissions of gas-and
  coal-fired power plants in a future with growing renewables.
\newblock {\em Renewable and Sustainable Energy Reviews}, 2017.

\bibitem{regulationdata}
http://www.pjm.com/markets-and-operations/ancillary-services.aspx.

\end{thebibliography}

\newpage
\section*{Appendix}

\subsection*{\bf{A: Proof of Theorem \ref{noncausalresourceprocurement}}}

\begin{proof}
Comparing (\ref{resourceprocurement}) with (\ref{resourceprocurement_linear}), it suffices to show that the polytope containment constraint (\ref{polytopecontain}) is equivalent to (\ref{conservationofall_linearcc}) and (\ref{feasibleset_icc}). As $E$ and the Minkowski sum of $\alpha_iS_i$ are both convex, (\ref{polytopecontain}) is equivalent to (\ref{intermediate}).  
Denote $q_{i,k}\in \mathbb{R}^T$ as the allocation of $v_k$ to the $i$th resource, then the constraint (\ref{intermediate}) can be written as follows:
\begin{align}
\label{intermediatecc2}
\begin{cases}
v_k=\sum_{i=1}^N q_{i,k}, \quad \forall k\in [K], \\
q_{i,k}\in \alpha_i S_i, \quad \forall k\in [K], i\in [N]. 
\end{cases}
\end{align} 
Clearly, (\ref{intermediatecc2}) is equivalent to (\ref{conservationofall_linearcc}) and (\ref{feasibleset_icc}). This completes the proof.
\end{proof}

\subsection*{\bf{B: Proof of Proposition \ref{solutiontoexample}}}
\begin{proof}
To prove that $\alpha_1^*  = \alpha_2^* = 1$, we first show that $J^*\geq 4$, and $\alpha_1=\alpha_2=1$ is the only possible solution that attains $J^*=4$. Second, we show that $\alpha_1=\alpha_2=1$ satisfies the polytope containment constraint (\ref{polytopecontainexample}). 

To cover $(1,1,4)$, the total maximum rate $\alpha_1 r_1+\alpha_2 r_2$ is at least 4 (otherwise covering $e^3$ is not possible). Therefore, a necessary condition is:
\begin{equation}
\label{condition1}
3\alpha_1 +\alpha_2 \geq 4.
\end{equation}
In addition, to cover $(1,1,4)$, the total capacity $\alpha_1C_1+\alpha_2C_2$ is at least $1+1+4$. Therefore, another necessary condition is:
\begin{equation}
\label{condition2}
3\alpha_1 +3\alpha_2 \geq 6.
\end{equation}
Combining (\ref{condition1}) and (\ref{condition2}), the optimal value to the following problem is an upper bound for $J^*$:
\begin{equation}
\label{upperbound4example}
\hspace{-0cm} \min_{\alpha_1,\alpha_2} 3\alpha_1+\alpha_2
\end{equation}
\begin{subnumcases}{\label{constraint4upperbound}}
3\alpha_1 +\alpha_2 \geq 4,   \\
3\alpha_1 +3\alpha_2 \geq 6.
\end{subnumcases}
The optimal value of (\ref{upperbound4example}) is $4$, and the unique solution is $\alpha_1=\alpha_2=1$. Therefore, $\alpha_1=\alpha_2=1$ is necessary. 

Next we show that $\alpha_1=\alpha_2=1$ is sufficient, i.e., the polytope containment constraint (\ref{polytopecontainexample}) is satisfied.  To see this, note that both $E$ and $S_1\oplus S_2$ are convex. Therefore, (\ref{polytopecontainexample}) is equivalent to the vertices of $E$ contained in $S_1 \oplus S_2$. This trivially holds for the vertex $(0,0,0)$ and the vertex $(1,1,-2)\in S_1$. The other vertex $(1,1,4)$ can be decomposed into $(0,0,3)\in S_1$ and $(1,1,1)\in S_2$. This completes our argument.

To show that the optimal asset mix is insufficient to causally cover $E$, we proceed as follows. Let $\phi_i^t(\cdot)$ be the allocation of signal $e^t$ to resource $i$. The first two elements of $(1,1,-2)$ and $(1,1,4)$ are the same. Since $e^t$ is revealed causally, the allocation for the first two periods should also be the same for $(1,1,-2)$ and $(1,1,4)$.  To cover the signal $(1,1,4)$, the unique allocation is to exclusively use the second battery for the first two periods, and then use both batteries at the third period, i.e., $\phi_1^1(\cdot)=\phi_1^2(\cdot)=0$, $\phi_1^3(\cdot)=3$, and $\phi_2^1(\cdot)=\phi_2^2(\cdot)=\phi_2^3(\cdot)=1$.  On the other hand, if we exclusively use the second battery for the first two periods, then $(1,1,-2)$ cannot be covered, since at time $3$, the maximum discharging rate of the second battery is $1<2$. This completes the proof. 
\end{proof}

\subsection*{\bf{D: Proof of Proposition \ref{meritorderproposition}}}
\begin{proof}
For each $i\in [N]$, define $k_i$ as the smallest $\alpha_i$ such that $E\subseteq \alpha_i S_i$. Assume that $k_i$ exits for all $i\in [N]$. This assumption has no loss of generality,  since if $k_i$ does not exist for some $i$,  then there is no $\alpha_i$ such that $\beta_i E \subseteq \alpha_iS_i$ when $\beta>0$. In this case,   the optimal solution to (\ref{causality4proportaional}) has to  satisfy $\alpha_i=\beta_i=0$. This means the resource $i$ is neither used nor procured, so we can remove it from $[N]$

Based on the definition of $k_i$, $E\subseteq \alpha_i S_i$ is  equivalent to $\alpha_i\geq k_i$, and the constraint (\ref{polytopecontain_causal_prop}) is equivalent to $\alpha_i\geq k_i\beta_i$. Plug this into (\ref{polytopecontain_causal_prop}), then the resource procurement problem (\ref{causality4proportaional}) becomes:
\begin{align}
\label{intermproposition5}
&\min_{\alpha_1,\ldots,\alpha_N, \beta} \sum_{i\in [N]} \alpha_i \pi_i  \\
&\begin{cases}
\alpha_i\geq k_i \beta_i, \quad \forall i\in [N]  \\ \nonumber
\sum_{i\in [N]} \beta_{i}=1, \\ \nonumber
0\leq \beta_i \leq 1, \forall i\in [N]. \nonumber 
\end{cases}
\end{align}
Clearly, at the optimal solution, we have $\alpha_i=k_i\beta_i$ for all $i\in [N]$. Then  (\ref{intermproposition5}) is equivalent to:
\begin{align*}
&\min_{\{\beta_i, i\in [N]\}} \sum_{i\in [N]} k_i \pi_i \beta_i \\
&\begin{cases}
\sum_{i\in [N]} \beta_{i}=1,\\ \nonumber
0\leq \beta_{i}\leq 1, \quad \forall i\in [N].
\end{cases}
\end{align*}
Rank all resource in the ascending order of $k_i\pi_i$.  At the optimal solution,  only resources with the smallest $k_i\pi_i$ is selected. This competes the proof.
\end{proof}

\subsection*{\bf{E: Proof of Theorem \ref{linearupperbound}}}
\begin{proof}
Since the allocation policy is affine, the resource procurement problem can be written as follows:
\begin{equation}
\label{resourceprocurement_causal_linear_intermediate}
\hspace{-4.5cm} \min \sum_{i=1}^{N} \alpha_i \pi_i
\end{equation}
\begin{subnumcases}{\label{constraint4causal_intermediate}}
e^t=\sum_{i=1}^N \left(C_i^t[e^1,\ldots,e^t]^T+D_i^t\right), \quad \forall t\in [T], \label{conservationofall_linear_1}\\
A_i[\phi_i^1(e^1),\ldots,\phi_i^T(e^{1:T})]^T\leq \alpha_iB_i, \forall i\in [N], \\
\alpha_i\geq 0,  \quad \forall i\in [N], \label{feasibleset_i_1}\\
\forall e\in E. \label{robustconstrinat}
\end{subnumcases}
where $\phi_i^t(\cdot)$ satisfies (\ref{affinepolicydef}). 
Note that constraints (\ref{conservationofall_linear_1})-(\ref{feasibleset_i_1}) is linear with respect to $e$. Therefore, (\ref{conservationofall_linear_1})-(\ref{feasibleset_i_1}) holds for all $e\in E$ if and only if it holds for all vertices of $E$. This is exactly  (\ref{constraint4causal}).
\end{proof}

\subsection*{\bf{F: Proof of Proposition \ref{poais1_3}}}
\begin{proof}
Let $S=\delta_iS_i$. The polytope containment constraint (\ref{polytopecontain}) becomes:
\begin{equation}
\label{identialrourceproofinter}
E\subseteq \dfrac{\alpha_1}{\delta_1}S \oplus \cdots \oplus \dfrac{\alpha_N}{\delta_N}S.  
\end{equation}
Let $\sigma_i=\alpha_i/\delta_i$, then (\ref{identialrourceproofinter}) is equivalent to:
\begin{equation}
\label{identialproofimport}
E\subseteq  \sum_{i=1}^N \sigma_i S.
\end{equation}
Consider the following allocation policy:
\begin{equation}
\label{allocationpolicyidential}
\phi_i^t(e^{1:t})=\dfrac{\sigma_i}{\sum_{i=1}^N \sigma_i} e^t.
\end{equation}
Clearly, (\ref{allocationpolicyidential}) is causal. To prove it covers all $e\in E$, we note that (\ref{identialproofimport}) is the same as $\dfrac{\sigma_i}{\sum_{i=1}^N \sigma_i} E \subseteq \sigma_i S$. Therefore:
\begin{equation*}
\dfrac{\sigma_i}{\sum_{i=1}^N \sigma_i}  e \in \sigma_i S=\alpha_i S_i, \quad \forall e\in E.
\end{equation*}
This completes the proof.
\end{proof}

\subsection*{\bf{G: Proof of Proposition \ref{poais1_2}}}
\begin{proof}
Let $\alpha^*=(\alpha_1^*,\ldots,\alpha_N^*)$ be the optimal solution to the resource procurement problem (\ref{resourceprocurement}). It suffices to construct a causal policy that covers all $e\in E$ under $\alpha^*$. To this end, define $\big(\phi_1^t(e^{1:t}), \ldots, \phi_N^t(e^{1:t})\big) $ as any  vector that:
\begin{align*}
\begin{cases}
e^t=\sum_{i=1}^N \phi_i^t(e^{1:t}), \quad \forall t\in [T], \\
\alpha_i^*\underline{s}_i^t\leq \phi_i^t(e^{1:t})  \leq \alpha_i^*\bar{s}_i^t, \quad \forall i\in [N], t\in [T].
\end{cases}
\end{align*}
The above allocation policy is causal, and it exists. If it does not exist,  then it indicates that there is some $t$ such that $e^t<\sum_{i=1}^N \alpha_i^*\underline{s}_i^t$ or $e^t>\sum_{i=1}^N \alpha_i^*\bar{s}_i^t$. This contradicts  (\ref{polytopecontain}).  For the same reason, it covers all $e\in E$. This completes the proof.
\end{proof}

\subsection*{\bf{H: Proof of Proposition \ref{poais1}}}
\begin{proof}
It suffices to construct an example with $PoC>1$. Consider to use two batteries to cover signals in $E=\{e|0\leq e^1 \leq 1, 0\leq e^2 \leq 1, -5 \leq e^3 \leq 7 \}$. The capacity of each battery is $C_1=9$ and $C_2=5$, and the maximum charging/discharging rate is  $r_1=2$, $r_2=5$. At time $0$, the initial state of charge is $33\%$ and $40\%$, respectively. Let $p_1=2$ and $p_2=5$. In this case, we can show that the optimal solution to (\ref{resourceprocurement}) is $\alpha_1=\alpha_2=1$. However, we can also show that when $\alpha_1=\alpha_2=1$, no causal policy can be found to cover both $(0,0,-5)$ and $(0,0,7)$. The proof is similar to that in Section III-A, and is therefore omitted.
\end{proof}

\subsection*{\bf{I: Proof of Theorem \ref{proposition4exactpoa}}}

To prove Theorem \ref{proposition4exactpoa}, we need the following lemma:
\begin{lemma}
\label{lemmaforbigtheor}
Under the assumptions of Theorem \ref{proposition4exactpoa}, there exist $T_0$ such that for all $T> T_0$,  there exists $(e^1,\ldots,e^{T-1})$ that satisfies $(e^1,\ldots,e^{T-1},\sum_i C_i/2) \in E$ and $(e^1,\ldots,e^{T-1},-\sum_i C_i/2) \in E$. 
\end{lemma}
\begin{proof}
since $E=S_1\oplus S_2 \oplus \cdots \oplus S_N$,  it suffices to construct a sequence $(e^1,\ldots, e^{T-1})$ and two allocation policies  to cover $(e^1,\ldots,e^{T-1},\sum_i C_i/2)$ and $(e^1,\ldots,e^{T-1},-\sum_i C_i/2)$ using unit batteries.
Without loss of generality,  assume that all batteries are empty at time $0$, i.e., $\theta_i=0$. If $\theta_i>0$, we can construct a sequence of signals to empty all batteries first and then concatenate this sequence with $(e^1,\ldots,e^{T-1})$.

Consider the first allocation policy to satisfy the following: 
\begin{itemize}
\item  Each battery $i$ is charged by no more than $r_i$ for all $t \leq T-1$, i.e.,  $\sum_{k=1}^t s_i^k\leq r^i$, $\forall t\leq T-1$,
\item The ($i+1$)th battery is used only if the $i$th battery is already charged by $r_i$, i.e., $s_i^t>0$ only if $s_j^t=r_j$ for all $j<i$.
\end{itemize}
Under this policy, the maximum total power the batteries can provide at time $t$ is $\sum_{i=1}^{N} r_i- \sum_{k=1}^{t-1} e^k$. On the other hand, consider a signal sequence  $(e^1,\ldots,e^{T-1})$  that satisfies:
\begin{align}
\label{constructsequence}
\begin{cases}
 \sum_{t=1}^{T-1} e^t=\dfrac{1}{2}\sum_{i=1}^N C_i, \\
 e^t\geq 0,  \quad \forall 0\leq t \leq T-1.
 \end{cases}
\end{align}
Clearly, for any $t\leq T-1$, as long as  $(e^1,\ldots,e^{T-1})$  satisfies (\ref{constructsequence}), we have
$e^t\leq \sum_{i=1}^N C_i/2-\sum_{k=1}^{t-1} e^k$. Since $\sum_{i=1}^{N} r_i\geq \sum_{i=1}^N C_i/2$, we have:
\begin{equation*}
e^t\leq \sum_{i=1}^N C_i/2-\sum_{k=1}^{t-1} e^k \leq \sum_{i=1}^{N} r_i- \sum_{k=1}^{t-1} e^k.
\end{equation*}
Therefore, the first policy covers all  $(e^1,\ldots,e^{T-1})$ that satisfies (\ref{constructsequence}). In addition, since each battery $i$ are charged by at most $r_i$ at the first $T-1$ steps, all batteries can be discharged to empty at time $T$. Therefore, it is able to cover $-\sum_{i=1}^N C_i/2$ at time $T$.

The second policy satisfies the following conditions: 
\begin{itemize}
\item  Each battery $i$ is charged by no more than $C_i-r_i$ for all $t \leq T-1$. Formally, $\sum_{k=1}^t s_i^k\leq C_i-r^i$, $\forall t\leq T-1$,
\item The ($i+1$)th battery is used only if the $i$th battery is already charged by $C_i-r_i$, i.e., $s_i^t>0$ only if $s_j^t=C_j-r_j$ for all $j<i$.
\end{itemize}
Trivially, under this policy, we can find a $T_0>0$ so that $(e^1,\ldots,e^{T_0})$ can be covered by the unit batteries and satisfies:
\begin{align*}
\begin{cases}
 \sum_{t=1}^{T_0} e^t=\dfrac{1}{2}\sum_{i=1}^N C_i, \\
 e^t\geq 0,  \quad \forall 0\leq t \leq T_0.
 \end{cases}
\end{align*}
For any $T> T_0$, let $(e^1,\ldots,e^{T-1})=(e^1,\ldots,e^{T_0}, 0,\ldots,0)$. Since $(e^1,\ldots,e^{T_0})$ is covered by unit batteries,  $(e^1,\ldots,e^{T-1})$ is also covered. 
In addition, since each battery $i$ is charged by at most $C_i-r_i$,  the policy  also guarantees to cover $\sum_{i=1}^N C_i/2$ at time $T$. 
This completes the proof of Lemma \ref{lemmaforbigtheor}.
\end{proof}

Using Lemma \ref{lemmaforbigtheor}, the proof of Theorem \ref{proposition4exactpoa} is as follows:

\begin{proof}
It suffices to show that (\ref{constraint4poaaslinearprogramming}) is both necessary and sufficient for  (\ref{constraint4causal_formualtion}). To prove necessity, we first note that (\ref{constraint1forpoaaslp}) is clearly necessary: if (\ref{constraint1forpoaaslp}) is not satisfied, we can construct some signal such that $e^t=\sum_{i=1}^N r_i$ for some $t$, then this signal can not be covered. To show that (\ref{constraint2forpoaaslp}) is also necessary, we note that according to Lemma \ref{lemmaforbigtheor}, there exist $T_0$ and $(e^1,\ldots,e^{T-1})$ such that $(e^1,\ldots,e^{T-1},\sum_i C_i/2) \in E$ and $(e^1,\ldots,e^{T-1},-\sum_i C_i/2) \in E$. Since the first $T-1$ elements of the signals are the same, a causal policy should make the same allocation decision for these two cases in the first $T-1$ steps. Therefore, at time $T$, the batteries should satisfy both $\sum_iC_i/2$ and $-\sum_iC_i/2$. Let $x_i^t$ denote the state of charge of battery $i$ at time $t$, i.e., $x_i^t=\theta_iC_i+\sum_{k=1}^t s_i^t$, then the following conditions should hold:
\begin{numcases}{}
\sum_{i=1}^N \alpha_i \min (r_i,C_i-x_i^{T-1} ) \geq \dfrac{1}{2} \sum_{i=1}^N C_i,  \label{conditions4proof1}\\
\sum_{i=1}^N \alpha_i \min (r_i,x_i^{T-1} ) \geq \dfrac{1}{2} \sum_{i=1}^N C_i. \label{conditions4proof2}
\end{numcases}
where $\min(r_i,C_i-x_i^{T-1})$ is the maximum energy battery $i$ can absorb at time $T$, and $\min(r_i,x_i^{T-1})$ is the maximum energy battery $i$ can provide at time $T$. 	Clearly, (\ref{conditions4proof1}) is necessary for the policy to cover $\sum_{i=1}^N C_i/2$ at time $T$, and (\ref{conditions4proof2}) is necessary for the policy to cover $-\sum_{i=1}^N C_i/2$ at time $T$. Adding up (\ref{conditions4proof1}) and (\ref{conditions4proof2}), and using the following inequality (can be verified by discussing different cases):
\begin{equation}
\min(r_i,C_i-x_i^{T-1})+\min(r_i,x_i^{T-1})\leq \min(2r_i,C_i),
\end{equation}
we obtain
(\ref{constraint2forpoaaslp}) as a necessary condition. Therefore,  (\ref{constraint4poaaslinearprogramming}) is necessary for  (\ref{constraint4causal_formualtion}).

To prove sufficiency, we construct a causal policy that covers all signals in $E$ when (\ref{constraint4poaaslinearprogramming})  holds.  Rank all batteries in a decreasing order of $C_i/r_i$. Without loss of generality, assume battery $1$ has the largest $C_i/r_i$, battery $2$ has the second largest, and so on. Divide $[N]$ into three subsets: $\mathcal{N}_1=\{i\in [N]|C_i=r_i \}$, $\mathcal{N}_2=\{i\in [N]|r_i<C_i<2r_i \}$, and $\mathcal{N}_3=\{i\in [N]|C_i\geq 2r_i \}$. 

To enhance readability, we first prove the result for a simple example, then we discuss how to generalize the proof to all other cases. As an example, consider  a group of three batteries with the following parameter: $C_1=r_1$, $C_2\geq 2r_2$, and  $C_3 \geq 2r_3$. We propose a causal allocation policy that divides each battery into a few blocks, and uses each block according to an alternating order to construct a virtual battery. The idea is illustrated in Figure \ref{allocation}, where the first battery contributes one block $\alpha_1 r_1$, the second battery contributes two identical blocks $\alpha_2 r_2$, and the third battery contributes two identical blocks $\alpha_3 r_3$. These blocks are arranged in the following order: $\alpha_3 r_3$, $\alpha_2 r_2$, $\alpha_1 r_1$, $\alpha_3 r_3$, $\alpha_2 r_2$, as illustrated in Figure \ref{allocation}. 
\begin{figure}[t]%
\centering
\includegraphics[width = 0.4\linewidth]{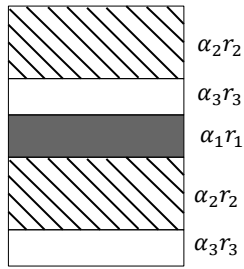}
\caption{The causal allocation: battery $1$ is one block $\alpha_1r_1$, battery $2$ is divided into 2 identical blocks $\alpha_2 r_2$, and battery $3$ is divided into two identical  blocks $\alpha_3 r_3$. These blocks are arranged in the following order: $\alpha_3r_3$, $\alpha_2r_2$, $\alpha_1r_1$, $\alpha_3r_3$, $\alpha_2r_2$.}
\label{allocation}
\end{figure}
Formally, the allocation in Figure \ref{allocation} can be described as one that respects the following conditions:
\begin{align}
\label{causalpolicy}
\begin{cases}
x_2^t= 0 \text{ unless } x_3^t=\alpha_3r_3, \\
x_1^t=0 \text{ unless } x_2^t\geq \alpha_2r_2  \text{ and } x_3^t\geq \alpha_3 r_3, \\
x_3^t\leq \alpha_3r_3 \text{ unless } x_1^t=\alpha_1C_1, \\
x_2^t\leq \alpha_2r_2 \text{ unless }  x_1^t=\alpha_1C_1 \text{ and }x_3^t=2\alpha_3r_3 , \\
x_2^t\leq 2\alpha_2r_2, \quad x_3^t\leq 2\alpha_3r_3.
\end{cases}
\end{align}
Clearly, this single battery has capacity $\alpha_1C_1+2\alpha_2r_2+2\alpha_2r_3$. Due to (\ref{constraint2forpoaaslp}), it is greater than $\sum_iC_i$. In addition, the charging/discharging rate is $\sum_i\alpha_ir_i$. Due to (\ref{constraint1forpoaaslp}), it is greater than $\sum_ir_i$. Therefore, it covers all signals in $E=S_1\oplus S_2 \oplus S_3$.

Next we generalize the proof idea to all possible cases. Note that in the example, $\mathcal{N}_1=\{1\}$, $\mathcal{N}_2=\emptyset$, and $\mathcal{N}_3=\{2,3\}$. We simply need to generalize the result to all possible combinations of $\mathcal{N}_1$, $\mathcal{N}_2$ and $\mathcal{N}_3$. First, we note that if there are more than one battery in $\mathcal{N}_1$, there is no difference, as these batteries can be used as a single one. Second, if there are more than two batteries in $\mathcal{N}_3$, then the proof idea of the example simply goes through. Third, if $\mathcal{N}_2$ is non-empty, i.e., $j\in \mathcal{N}_2$, then we remove battery $j$ from $[N]$, divide it into two batteries, and then place these two batteries back to $[N]$. In particular, denote these two new batteries as $i=N+1$ and $i=N+2$, such that $\alpha_{N+1}=\alpha_{N+2}=\alpha_j$, $C_{N+1}=r_{N+1}=2r_i-C_i$ and $C_{N+2}=2C_i-2r_i$, $r_{N+2}=C_i-r_i$. Note that $C_{N+1}+C_{N+2}=C_{j}$ and $r_{N+1}+r_{N+2}=r_j$. Therefore, it is a division of battery $j$.  
After this operation, we have $N+1\in \mathcal{N}_1$ and $N+2\in \mathcal{N}_3$.  Apply this operation for all  $i\in \mathcal{N}_2$ until $\mathcal{N}_2=\emptyset$, then we are back to the case of the example. This completes the proof.
\end{proof}

\begin{IEEEbiography}
    [{\includegraphics[width=1in,height=1.25in,clip,keepaspectratio]{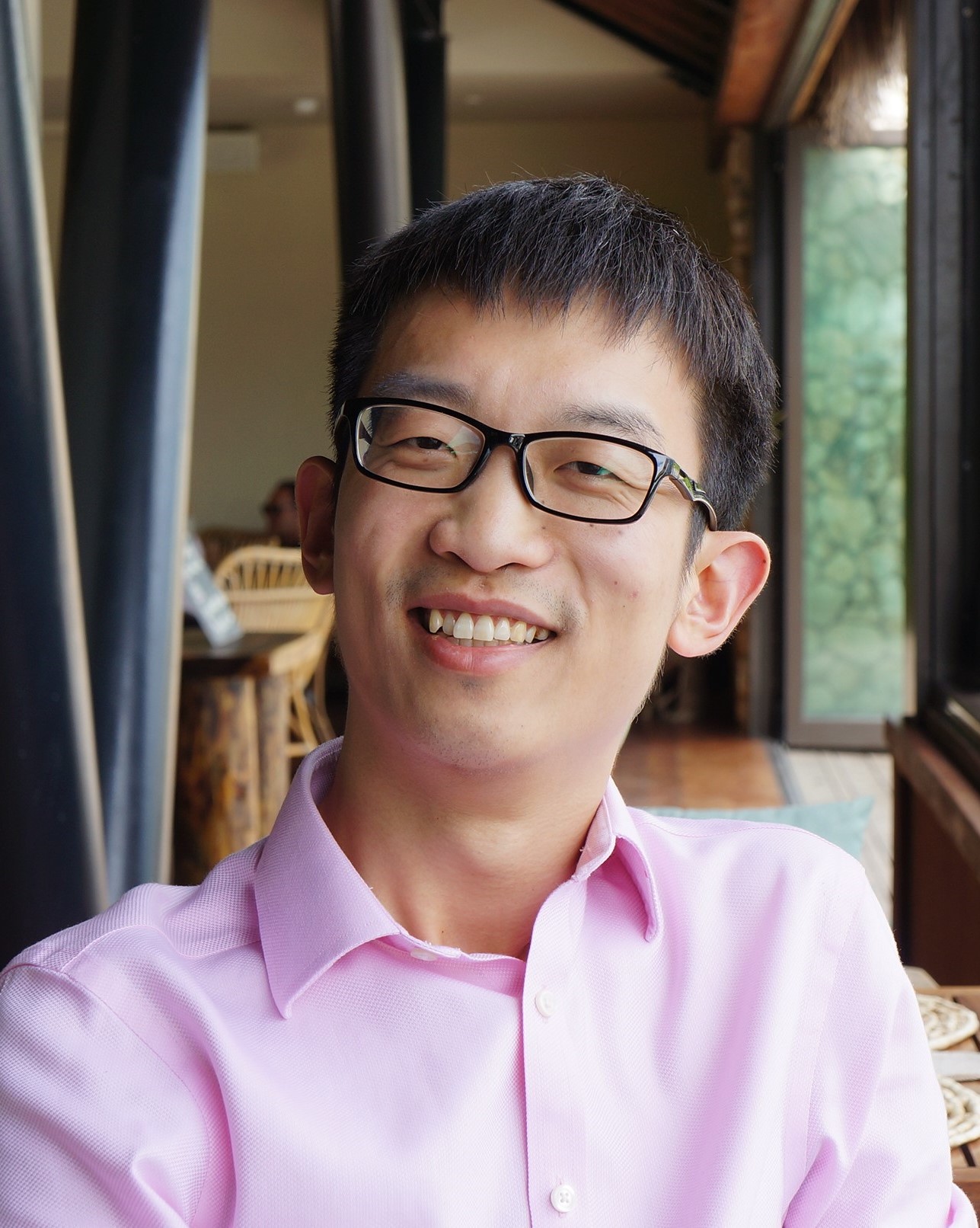}}]{Sen Li} is a postdoctoral fellow in the Department of Mechanical Engineering at University of California, Berkeley, working with Prof. Kameshwar Poolla and Prof. Pravin Varaiya. He received his B.S. from Zhejiang University, and Ph.D. from The Ohio State University. Previously, Dr. Li was an intern at Pacific Northwestern National Laboratory, and a visiting student at Harvard University. His research lies in the intersection of control, optimization and game theory with applications in large-scale cyber-physical systems. He is particularly interested in renewable energy integration and intelligent transportation systems. He is a finalist of Best Student Paper Award at 2018 European Control Conference.
\end{IEEEbiography}
\begin{IEEEbiography}
[{\includegraphics[width=1in,height=1.25in,clip,keepaspectratio]{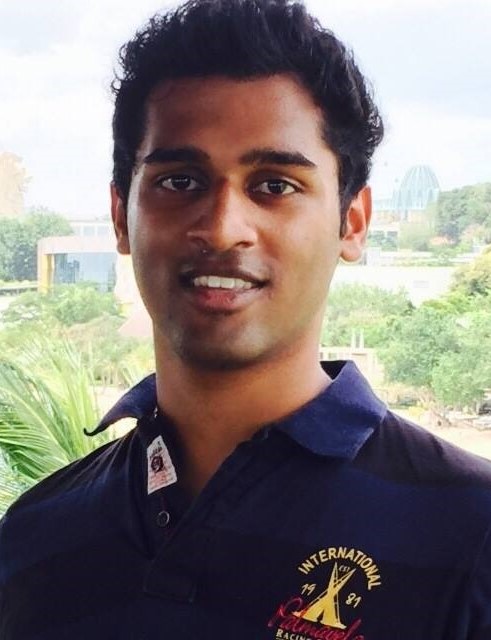}}]{Akhil Shetty }
is currently a Ph. D. student in Electrical Engineering and Computer Sciences at University of California, Berkeley. He received both his B. Tech and M. Tech degrees in Electrical Engineering from the Indian Institute of Technology, Bombay in 2017.  His current research interests center on developing control, optimization and statistical tools to solve problems in energy systems, traffic safety and transportation economics. He is a recipient of the Berkeley Graduate Fellowship for 2017-19. He was a Best Student Paper Finalist at ECC 2018 and received the Best Student Paper Award at IEEE SPCOM 2014.
\end{IEEEbiography}
\begin{IEEEbiography}
    [{\includegraphics[width=1in,height=1.25in,clip,keepaspectratio]{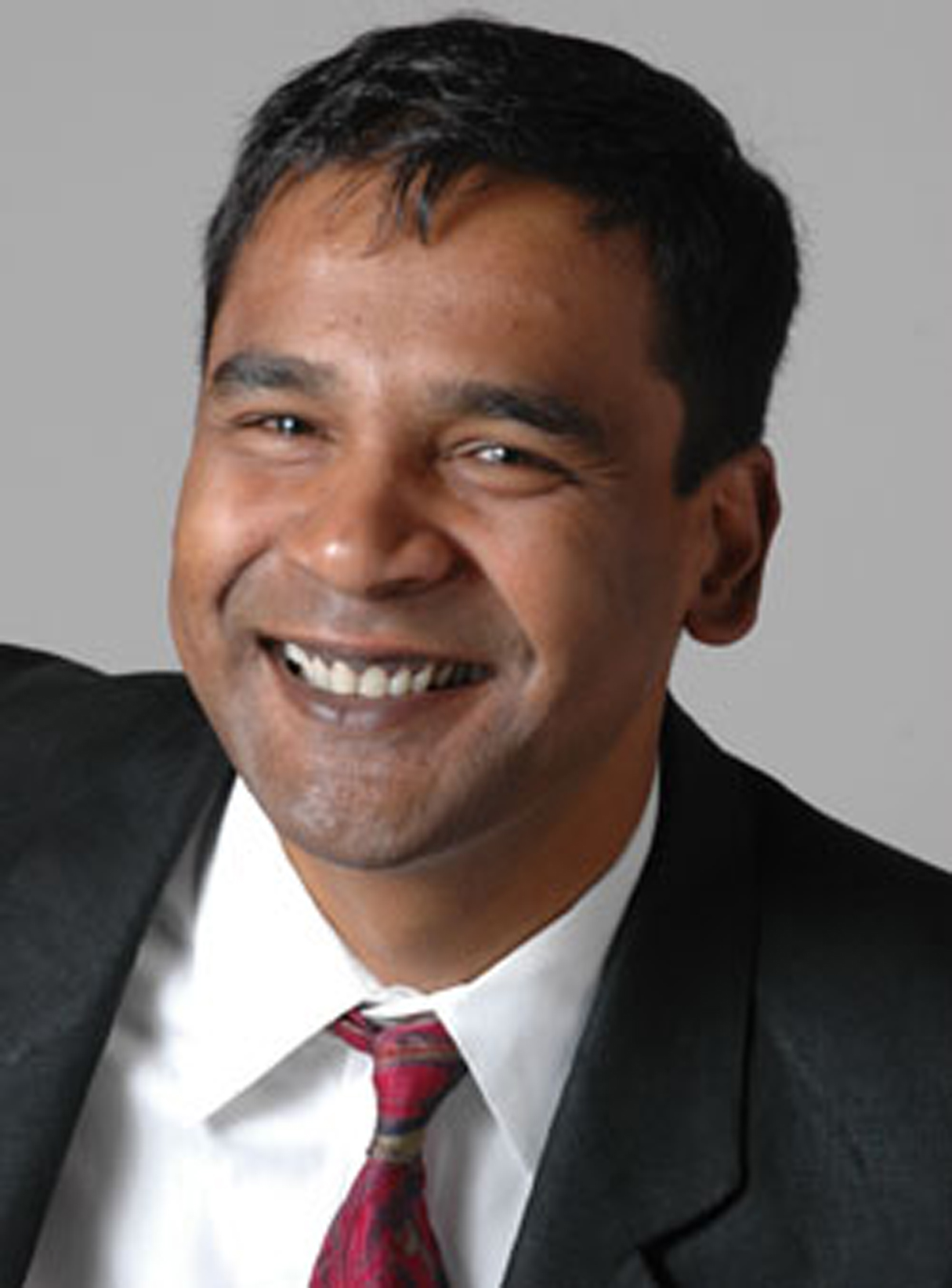}}]
  {Kameshwar Poolla} is the Cadence Design Systems Distinguished Professor, in Department of Electrical Engineering and Computer Sciences and Department of Mechanical Engineering at University of California, Berkeley. He received his B.Tech. degree from the Indian Institute of Technology, Bombay in 1980 and his Ph.D. from the Center for Mathematical System Theory, University of Florida, Gainesville in 1984. Poolla's research interests include many aspects of future energy systems including economics, security, and commercialization. He was awarded the 1984 Outstanding Dissertation Award by the University of Florida; the 1988 NSF Presidential Young Investigator Award; the 1993 Hugo Schuck Best Paper Prize; the 1994 Donald P. Eckman Award; a 1997 JSPS Fellowship; the 1997 Distinguished Teaching Award from the University of California, Berkeley; and the 2004 and 2007 IEEE Transactions on Semiconductor Manufacturing Best Paper Prizes; and the 2009 IEEE CSS Transition to Practice Award.
\end{IEEEbiography}

\begin{IEEEbiography}
    [{\includegraphics[width=1in,height=1.25in,clip,keepaspectratio]{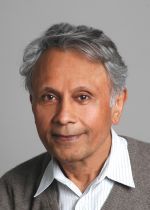}}]
  {Pravin Varaiya} is the Nortel Networks Distinguished Professor in the Department of Electrical Engineering and Computer Sciences at the UC Berkeley. From 1975 to 1992 he was also Professor of Economics at Berkeley. From 1994 to 1997 he was Director of the California PATH program, a multi-university research program dedicated to the solution of California's transportation problems.
Varaiya has held a Guggenheim Fellowship and a Miller Research Professorship. He received Honorary Doctorates from L'Institut National Polytechnique de Toulouse and L'Institut National Polytechnique de Grenoble. He received the IEEE Control Systems Society Field Award, the Hendrik W. Bode Lecture Prize, the Richard E. Bellman Control Heritage Award, the IEEE ITS Society Outstanding Research Award, and the IEEE ITS Lifetime Achievement Award. He is a Fellow of IEEE, a member of the National Academy of Engineering, and a Fellow of the American Academy of Arts and Sciences.
\end{IEEEbiography}

\end{document}